\documentclass{amsart}
\usepackage{amssymb}
\usepackage{amscd}
\usepackage{verbatim}
\usepackage{epsfig}

\begin{document}
\newcommand\Mand{\ \text{and}\ }
\newcommand\Mor{\ \text{or}\ }
\newcommand\Mfor{\ \text{for}\ }
\newcommand\Real{\mathbb{R}}
\newcommand\RR{\mathbb{R}}
\newcommand\im{\operatorname{Im}}
\newcommand\re{\operatorname{Re}}
\newcommand\sign{\operatorname{sign}}
\newcommand\sphere{\mathbb{S}}
\newcommand\BB{\mathbb{B}}
\newcommand\HH{\mathbb{H}}
\newcommand\ZZ{\mathbb{Z}}
\newcommand\codim{\operatorname{codim}}
\newcommand\Sym{\operatorname{Sym}}
\newcommand\End{\operatorname{End}}
\newcommand\Span{\operatorname{span}}
\newcommand\Ran{\operatorname{Ran}}
\newcommand\ep{\epsilon}
\newcommand\Cinf{\cC^\infty}
\newcommand\dCinf{\dot \cC^\infty}
\newcommand\CI{\cC^\infty}
\newcommand\dCI{\dot \cC^\infty}
\newcommand\Cx{\mathbb{C}}
\newcommand\Nat{\mathbb{N}}
\newcommand\dist{\cC^{-\infty}}
\newcommand\ddist{\dot \cC^{-\infty}}
\newcommand\pa{\partial}
\newcommand\Card{\mathrm{Card}}
\renewcommand\Box{{\square}}
\newcommand\WF{\mathrm{WF}}
\newcommand\WFh{\mathrm{WF}_\semi}
\newcommand\WFb{\mathrm{WF}_\bl}
\newcommand\Vf{\mathcal{V}}
\newcommand\Vb{\mathcal{V}_\bl}
\newcommand\Vz{\mathcal{V}_0}
\newcommand\Hb{H_{\bl}}
\newcommand\Ker{\mathrm{Ker}}
\newcommand\Range{\mathrm{Ran}}
\newcommand\Hom{\mathrm{Hom}}
\newcommand\Id{\mathrm{Id}}
\newcommand\sgn{\operatorname{sgn}}
\newcommand\ff{\mathrm{ff}}
\newcommand\tf{\mathrm{tf}}
\newcommand\supp{\operatorname{supp}}
\newcommand\vol{\mathrm{vol}}
\newcommand\Diff{\mathrm{Diff}}
\newcommand\Diffd{\mathrm{Diff}_{\dagger}}
\newcommand\Diffs{\mathrm{Diff}_{\sharp}}
\newcommand\Diffb{\mathrm{Diff}_\bl}
\newcommand\DiffbI{\mathrm{Diff}_{\bl,I}}
\newcommand\Diffbeven{\mathrm{Diff}_{\bl,\even}}
\newcommand\Diffz{\mathrm{Diff}_0}
\newcommand\Psih{\Psi_{\semi}}
\newcommand\Psihcl{\Psi_{\semi,\cl}}
\newcommand\Psib{\Psi_\bl}
\newcommand\Psibc{\Psi_{\mathrm{bc}}}
\newcommand\TbC{{}^{\bl,\Cx} T}
\newcommand\Tb{{}^{\bl} T}
\newcommand\Sb{{}^{\bl} S}
\newcommand\Lambdab{{}^{\bl} \Lambda}
\newcommand\zT{{}^{0} T}
\newcommand\Tz{{}^{0} T}
\newcommand\zS{{}^{0} S}
\newcommand\dom{\mathcal{D}}
\newcommand\cA{\mathcal{A}}
\newcommand\cB{\mathcal{B}}
\newcommand\cE{\mathcal{E}}
\newcommand\cG{\mathcal{G}}
\newcommand\cH{\mathcal{H}}
\newcommand\cU{\mathcal{U}}
\newcommand\cO{\mathcal{O}}
\newcommand\cF{\mathcal{F}}
\newcommand\cM{\mathcal{M}}
\newcommand\cQ{\mathcal{Q}}
\newcommand\cR{\mathcal{R}}
\newcommand\cI{\mathcal{I}}
\newcommand\cL{\mathcal{L}}
\newcommand\cK{\mathcal{K}}
\newcommand\cC{\mathcal{C}}
\newcommand\cX{\mathcal{X}}
\newcommand\cY{\mathcal{Y}}
\newcommand\cP{\mathcal{P}}
\newcommand\cS{\mathcal{S}}
\newcommand\cZ{\mathcal{Z}}
\newcommand\cW{\mathcal{W}}
\newcommand\Ptil{\tilde P}
\newcommand\ptil{\tilde p}
\newcommand\chit{\tilde \chi}
\newcommand\yt{\tilde y}
\newcommand\zetat{\tilde \zeta}
\newcommand\xit{\tilde \xi}
\newcommand\taut{{\tilde \tau}}
\newcommand\phit{{\tilde \phi}}
\newcommand\mut{{\tilde \mu}}
\newcommand\sigmah{\hat\sigma}
\newcommand\zetah{\hat\zeta}
\newcommand\etah{\hat\eta}
\newcommand\loc{\mathrm{loc}}
\newcommand\compl{\mathrm{comp}}
\newcommand\reg{\mathrm{reg}}
\newcommand\GBB{\textsf{GBB}}
\newcommand\GBBsp{\textsf{GBB}\ }
\newcommand\bl{{\mathrm b}}
\newcommand{\sH}{\mathsf{H}}
\newcommand{\cte}{\digamma}
\newcommand\cl{\operatorname{cl}}
\newcommand\hsf{\mathcal{S}}
\newcommand\Div{\operatorname{div}}
\newcommand\hilbert{\mathfrak{X}}

\newcommand\Hh{H_{\semi}}

\newcommand\bM{\bar M}
\newcommand\Xext{X_{-\delta_0}}

\newcommand\xib{{\underline{\xi}}}
\newcommand\etab{{\underline{\eta}}}
\newcommand\zetab{{\underline{\zeta}}}

\newcommand\xibh{{\underline{\hat \xi}}}
\newcommand\etabh{{\underline{\hat \eta}}}
\newcommand\zetabh{{\underline{\hat \zeta}}}

\newcommand\zn{z}
\newcommand\sigman{\sigma}
\newcommand\psit{\tilde\psi}
\newcommand\rhot{{\tilde\rho}}

\newcommand\hM{\hat M}

\newcommand\Op{\operatorname{Op}}
\newcommand\Oph{\operatorname{Op_{\semi}}}

\newcommand\innr{{\mathrm{inner}}}
\newcommand\outr{{\mathrm{outer}}}
\newcommand\full{{\mathrm{full}}}
\newcommand\semi{\hbar}

\newcommand\elliptic{\mathrm{ell}}
\newcommand\diffordgen{k}
\newcommand\difford{2}
\newcommand\diffordm{1}
\newcommand\diffordmpar{1}
\newcommand\even{\mathrm{even}}
\newcommand\dimn{n}
\newcommand\dimnpar{n}
\newcommand\dimnm{n-1}
\newcommand\dimnp{n+1}
\newcommand\dimnppar{(n+1)}
\newcommand\dimnppp{n+3}
\newcommand\dimnppppar{n+3}

\setcounter{secnumdepth}{3}
\newtheorem{lemma}{Lemma}[section]
\newtheorem{prop}[lemma]{Proposition}
\newtheorem{thm}[lemma]{Theorem}
\newtheorem{cor}[lemma]{Corollary}
\newtheorem{result}[lemma]{Result}
\newtheorem*{thm*}{Theorem}
\newtheorem*{prop*}{Proposition}
\newtheorem*{cor*}{Corollary}
\newtheorem*{conj*}{Conjecture}
\numberwithin{equation}{section}
\theoremstyle{remark}
\newtheorem{rem}[lemma]{Remark}
\newtheorem*{rem*}{Remark}
\theoremstyle{definition}
\newtheorem{Def}[lemma]{Definition}
\newtheorem*{Def*}{Definition}

\newcommand{\mar}[1]{{\marginpar{\sffamily{\scriptsize #1}}}}
\newcommand\av[1]{\mar{AV:#1}}

\renewcommand{\theenumi}{\roman{enumi}}
\renewcommand{\labelenumi}{(\theenumi)}

\title[Microlocal analysis of asymptotically hyperbolic spaces]{Microlocal analysis of
asymptotically hyperbolic spaces and high energy resolvent estimates}
\author[Andras Vasy]{Andr\'as Vasy}
\address{Department of Mathematics, Stanford University, CA 94305-2125, USA}

\email{andras@math.stanford.edu}


\date{May 30, 2011. Original version: April 7, 2011.}
\thanks{The author gratefully
  acknowledges partial support from the NSF under grant number  
  DMS-0801226 and from a Chambers Fellowship at Stanford University, as well
as the hospitality of MSRI in Berkeley in Fall 2010.}

\begin{abstract}
In this paper
we describe a new method for analyzing the Laplacian on
asymptotically hyperbolic spaces, which was introduced in
\cite{Vasy-Dyatlov:Microlocal-Kerr}. This new method in particular
constructs the analytic continuation of the resolvent for even metrics
(in the sense of Guillarmou), and gives high energy estimates in strips.
The key idea is an extension across the boundary for a problem
obtained from the Laplacian shifted by the spectral parameter. The
extended problem is non-elliptic -- indeed, on the other side it is
related to the Klein-Gordon equation on an asymptotically de Sitter
space -- but nonetheless it can be analyzed by methods of Fredholm
theory. This method is a special case of a more general approach to the
analysis of PDEs which includes, for instance, Kerr-de Sitter and
Minkowski type spaces; see \cite{Vasy-Dyatlov:Microlocal-Kerr} for
details.
The present paper is self-contained, and deals with asymptotically
hyperbolic spaces
without burdening the reader with material only needed for the
analysis of the Lorentzian problems considered in \cite{Vasy-Dyatlov:Microlocal-Kerr}.
\end{abstract}

\maketitle

\section{Introduction}
In this paper
we describe a new method for analyzing the Laplacian on
asymptotically hyperbolic, or conformally compact, spaces,
which was introduced in
\cite{Vasy-Dyatlov:Microlocal-Kerr}. This new method in particular
constructs the analytic continuation of the resolvent for even metrics
(in the sense of Guillarmou \cite {Guillarmou:Meromorphic}), and gives high energy estimates in strips.
The key idea is an extension across the boundary for a problem
obtained from the Laplacian shifted by the spectral parameter. The
extended problem is non-elliptic -- indeed, on the other side it is
related to the Klein-Gordon equation on an asymptotically de Sitter
space -- but nonetheless it can be analyzed by methods of Fredholm
theory. In \cite{Vasy-Dyatlov:Microlocal-Kerr} these methods, with
some additional ingredients, were
used to analyze the wave equation on Kerr-de Sitter space-times; the
present setting is described there as the simplest application of the
tools introduced. The purpose of the present paper is to give a
self-contained treatment of conformally compact spaces, without
burdening the reader with the additional machinery required for the
Kerr-de Sitter analysis.

We start by recalling the definition of
manifolds with
{\em even} conformally compact metrics.
These are Riemannian metrics $g_0$
on the interior of an $n$-dimensional compact manifold with boundary
$X_0$ such that near the boundary $Y$,
with a product decomposition nearby and
a defining function $x$, they are
of the form
$$
g_0=\frac{dx^2+h}{x^2},
$$
where $h$ is a family of metrics on $Y=\pa X_0$ depending on $x$ in an even manner,
i.e.\ only even powers of $x$ show up in the Taylor series. (There is a much more
natural way to phrase the evenness condition, see
\cite[Definition~1.2]{Guillarmou:Meromorphic}.) We also write $X_{0,\even}$ for
the manifold $X_0$ when the smooth structure has been changed so that
$x^2$ is a boundary defining function; thus, a smooth function on $X_0$ is
even if and only if it is smooth when regarded as a function on $X_{0,\even}$.
The analytic continuation of the resolvent in this category (but without the
evenness condition) was obtained
by Mazzeo and Melrose \cite{Mazzeo-Melrose:Meromorphic} (Agmon
\cite{Agmon:Spectral} and Perry
\cite{Perry:Laplace-I,Perry:Laplace-II}
had similar results in the
restricted setting of hyperbolic quotients), with the
possibility of
some essential singularities at pure imaginary half-integers noticed by Borthwick and Perry
\cite{Borthwick-Perry:Scattering}.
Guillarmou
\cite{Guillarmou:Meromorphic} showed that for even metrics the latter do not
exist, but generically they do exist for non-even metrics, by
a more careful analysis utilizing the work of Graham and Zworski
\cite{Graham-Zworski:Scattering}. Further, if the manifold
is actually asymptotic to hyperbolic space (note that hyperbolic space is of this form
in view of the Poincar\'e model),
Melrose, S\'a Barreto and Vasy \cite{Melrose-SaBarreto-Vasy:Semiclassical} proved
high energy resolvent estimates in strips around the real axis via a parametrix
construction; these are exactly
the estimates that allow expansions for solutions of the wave equation in terms
of resonances. Estimates just on the real axis were obtained by
Cardoso and Vodev for more general conformal infinities
\cite{Cardoso-Vodev:Uniform, Vodev:Local}.
One implication of our methods is a generalization of these results:
we allow general conformal infinities, and obtain estimates in
arbitrary strips.

Below $\dCI(X_0)$ denotes `Schwartz functions' on $X_0$, i.e.\ $\CI$ functions
vanishing with all derivatives at $\pa X_0$, and $\dist(X_0)$ is the dual space
of `tempered distributions' (these spaces are naturally identified for
$X_0$ and $X_{0,\even}$), while $H^s(X_{0,\even})$ is the standard Sobolev
space on $X_{0,\even}$ (corresponding
to extension across the boundary, see e.g.\ \cite[Appendix~B]{Hor}, where these
are denoted by $\bar H^s(X_{0,\even}^\circ)$). For instance,
$\|u\|_{H^{1}(X_{0,\even})}^2=\|u\|_{L^2(X_{0,\even})}^2
+\|d u\|^2_{L^2(X_{0,\even})}$, with the norms taken with
respect to any {\em smooth} Riemannian metric on $X_{0,\even}$ (all
choices yield equivalent norms by compactness). Here we point out that
while $x^2g_0$ is a smooth non-degenerate section of the pull-back of $T^*X_0$ to
$X_{0,\even}$ (which essentially means that it is a smooth, in $X_{0,\even}$,
non-degenerate linear combination of $dx$ and $dy_j$ in local
coordinates), as $\mu=x^2$ means $d\mu=2x\,dx$, it is actually
not a smooth section of $T^*X_{0,\even}$. However, $x^{n+1}\,|dg_0|$
is a smooth non-degenerate density, so $L^2(X_{0,\even})$ (up to norm
equivalence) is the $L^2$ space given by the density
$x^{n+1}\,|dg_0|$, i.e.\ is $x^{-(n+1)/2}L^2_{g_0}(X_0)$, i.e.
$$
\|x^{-(n+1)/2}u\|_{L^2(X_{0,\even})}\sim \|u\|_{L^2_{g_0}(X)}.
$$
Further, in local coordinates $(\mu,y)$,
using $2\pa_\mu=x^{-1}\pa_x$, the $H^1(X_{0,\even})$
norm of $u$ is equivalent to
$$
\|u\|_{L^2(X_{0,\even})}^2
+\|x^{-1}\pa_x
u\|^2_{L^2(X_{0,\even})}+\sum_{j=1}^{n-1}\|\pa_{y_j}u\|^2_{L^2(X_{0,\even})}.
$$

We also let $\Hh^s(X_{0,\even})$ be the standard
semiclassical Sobolev space, i.e.\ for $h$ bounded away from $0$ this
is equipped with a norm
equivalent to the standard fixed ($h$-independent) norm on
$H^s(X_{0,\even})$,
but the uniform behavior as $h\to 0$ is different; e.g.\ locally the
$\Hh^1(X)$ norm is given by $\|u\|_{\Hh^1}^2=\sum_{j} \|hD_j u\|^2_{L^2}+\|u\|^2_{L^2}$,
see \cite{Dimassi-Sjostrand:Spectral,
  Evans-Zworski:Semiclassical}. Thus, in \eqref{eq:intro-nontrap}, for
$s=1$ (which is possible when $C<1/2$, i.e.\ if one only considers the continuation into a
small strip beyond the continuous spectrum),
\begin{equation*}\begin{split}
s=1\qquad\Longrightarrow\qquad&\|u\|_{H^{s-1}_{|\sigma|^{-1}}(X_{0,\even})}=\|u\|_{L^2(X_{0,\even})}\\
&\Mand\|u\|_{H^{s}_{|\sigma|^{-1}}(X_{0,\even})}^2=\|u\|_{L^2(X_{0,\even})}^2
+|\sigma|^{-2}\|d u\|^2_{L^2(X_{0,\even})},
\end{split}\end{equation*}
with the norms taken with
respect to any smooth Riemannian metric on $X_{0,\even}$.

\begin{thm*}(See Theorem~\ref{thm:conf-compact-high} for the full statement.)
Suppose that $X_0$ is an $\dimnpar$-dimensional
manifold with boundary $Y$ with
an even Riemannian conformally compact metric $g_0$. Then
the inverse of
$$
\Delta_{g_0}-\left(\frac{\dimnm}{2}\right)^2-\sigma^2,
$$
written as $\cR(\sigma):L^2\to L^2$,
has a meromorphic continuation from
$\im\sigma\gg0$ to $\Cx$,
$$
\cR(\sigma):\dCI(X_0)\to\dist(X_0),
$$
with poles with finite rank residues. If in addition
$(X_0,g_0)$ is non-trapping, then
non-trapping estimates hold in every strip $|\im\sigma|<C$, $|\re\sigma|\gg 0$:
for
$s>\frac{1}{2}+C$,
\begin{equation}\label{eq:intro-nontrap}
\|x^{-(\dimnm)/2+\imath\sigma} \cR(\sigma)f\|_{H^s_{|\sigma|^{-1}}(X_{0,\even})}
\leq \tilde C|\sigma|^{-1}\|x^{-(\dimnppp)/2+\imath\sigma}f\|_{H^{s-1}_{|\sigma|^{-1}}(X_{0,\even})}.
\end{equation}
If $f$ has compact support in $X_0^\circ$,
the $s-1$ norm on $f$ can be replaced by the $s-2$ norm. For suitable
$\delta_0>0$, the estimates
are valid in regions $-C<\im\sigma<\delta_0|\re\sigma|$ if the multipliers
$x^{\imath\sigma}$
are slightly adjusted.
\end{thm*}

Further, as stated in Theorem~\ref{thm:conf-compact-high}, the resolvent
is {\em semiclassically outgoing} with a loss of $h^{-1}$, in the sense
of recent results of Datchev and Vasy \cite{Datchev-Vasy:Gluing-prop}
and \cite{Datchev-Vasy:Trapped}. This means that for mild trapping (where,
in a strip near the spectrum,
one has polynomially bounded resolvent for a compactly localized version of
the trapped model) one obtains resolvent bounds of the same kind as for the
above-mentioned trapped models, and lossless estimates microlocally
away from the trapping. In particular, one obtains logarithmic losses compared
to non-trapping on the spectrum for hyperbolic trapping in
the sense of \cite[Section~1.2]{Wunsch-Zworski:Resolvent}, and
polynomial losses in strips, since
for the compactly localized model this was
recently shown by Wunsch and Zworski \cite{Wunsch-Zworski:Resolvent}.

Our method is to change the smooth structure, replacing $x$ by $\mu=x^2$,
conjugate the operator by an appropriate weight as well as remove a vanishing
factor of $\mu$, and show that the new operator continues smoothly and
non-degenerately (in an appropriate sense) across $\mu=0$, i.e.\ $Y$,
to a (non-elliptic)
problem which we can analyze utilizing
by now almost standard tools of microlocal analysis. These steps are reflected in
the form of the estimate \eqref{eq:intro-nontrap}; $\mu$ shows up in
the use of evenness,
conjugation due to the presence of $x^{-\dimnppar/2+\imath\sigma}$, and the two halves
of the vanishing factor of $\mu$ being removed in $x^{\pm 1}$ on the left and
right hand sides.

While it might seem somewhat ad hoc, this construction in fact has
origins in wave propagation in one higher dimensional (i.e.\
$n+1$-dimensional)
Lorentzian spaces -- either Minkowski space,
or de Sitter space blown up at a point at future infinity. Namely in
both cases the wave equation (and the Klein-Gordon equation on de Sitter space)
is a totally characteristic, or b-, PDE, and
after a Mellin transform this gives a PDE on the sphere at infinity
in the Minkowski case, and on the front face of the blow-up in the de
Sitter setting. These are exactly the PDE arising by the process
described in the previous paragraph, with the original manifold $X_0$
lying in the interior of the light cone in Minkowski space (so there
are two copies, at future and past infinity) and in the interior of
the backward light cone from the blow-up point in the de Sitter case;
see
\cite{Vasy-Dyatlov:Microlocal-Kerr} for more detail.
This relationship, restricted to the $X_0$-region,
was exploited in \cite[Section~7]{Vasy:De-Sitter},
where the work of Mazzeo and Melrose was used to construct the Poisson
operator on asymptotically de Sitter spaces. Conceptually the main
novelty here is that we work directly with the extended problem, which
turns out to simplify the analysis of Mazzeo and Melrose in many ways and give a new
explanation for Guillarmou's results as well as yield high energy estimates.

We briefly describe this extended operator, $P_\sigma$.
It has radial points at the 
conormal bundle $N^*Y\setminus o$ 
of $Y$ 
in the sense of microlocal analysis, i.e.\ the Hamilton vector field is radial 
at these points, i.e.\ is 
a multiple of the generator of dilations of the fibers of the cotangent bundle
there. However, tools exist to deal with these, going back to Melrose's
geometric treatment of scattering theory on asymptotically Euclidean
spaces \cite{RBMSpec}. Note that $N^*Y\setminus o$ consists of two components,
$\Lambda_+$, resp.\ $\Lambda_-$, and in $S^*X=(T^*X\setminus o)/\RR^+$
the images, $L_+$, resp.\ $L_-$, of these
are sources, resp.\ sinks, for the Hamilton flow. At $L_\pm$ one has choices regarding
the direction one wants to propagate estimates (into or out of the radial points),
which directly correspond to working with strong or weak Sobolev spaces.
For the present problem, the relevant choice is propagating estimates {\em away from}
the radial points, thus working with the `good' Sobolev spaces (which can be
taken to have as positive order as one wishes; there is a minimum amount of
regularity imposed by our choice of propagation direction, cf.\ the requirement
$s>\frac{1}{2}+C$ above \eqref{eq:intro-nontrap}).
All other points are either elliptic, or microhyperbolic.
It remains to either deal with the non-compactness of the `far end' of the
$\dimnpar$-dimensional
de Sitter space --- or instead, as is indeed more convenient when one wants to
deal with more singular geometries, adding complex absorbing potentials,
in the spirit of works of Nonnenmacher and Zworski
\cite{Nonnenmacher-Zworski:Quantum} and Wunsch and Zworski
\cite{Wunsch-Zworski:Resolvent}. In fact, the complex absorption could be
replaced by adding a space-like boundary, see \cite{Vasy-Dyatlov:Microlocal-Kerr},
but for many microlocal purposes
complex absorption is more desirable, hence we follow the latter method.
However, crucially, these complex absorbing techniques (or the addition
of a space-like boundary) already enter
in the non-semiclassical problem in our case, as we are in a non-elliptic setting.

One can reverse the direction of the
argument and analyze the wave equation on an $\dimnpar$-dimensional even
asymptotically
de Sitter space $X_0'$ by extending it across the boundary, much like the
the Riemannian conformally compact
space $X_0$ is extended in this approach. Then, performing microlocal propagation
in the opposite direction, which amounts to working with the adjoint operators
that we already need in order to prove existence of solutions for the Riemannian
spaces,
we obtain existence, uniqueness and structure results for asymptotically
de Sitter spaces, recovering a large
part
of the results of \cite{Vasy:De-Sitter}. Here we only briefly indicate this method
of analysis in Remark~\ref{rem:asymp-dS}.

In other words, we establish a Riemannian-Lorentzian duality, that will have
counterparts both in the pseudo-Riemannian setting of higher signature and in
higher rank symmetric spaces, though in the latter the analysis might become
more complicated. Note that asymptotically hyperbolic and de Sitter spaces
are not connected by a `complex rotation' (in the sense of an actual deformation);
they are smooth continuations of each other in the sense we just discussed.

To emphasize the simplicity of our method, we list all of the microlocal techniques
(which are relevant both in the classical and in the semiclassical setting)
that we use on a {\em compact manifold without boundary}; in all cases {\em only
microlocal Sobolev estimates} matter (not parametrices, etc.):
\begin{enumerate}
\item
Microlocal elliptic regularity.
\item
Microhyperbolic propagation of singularities.
\item
{\em Rough} analysis at a Lagrangian invariant under the Hamilton flow
which roughly behaves like a collection of
radial points, though the internal structure does not matter, in the spirit
of \cite[Section~9]{RBMSpec}.
\item
Complex absorbing `potentials' in the spirit of
\cite{Nonnenmacher-Zworski:Quantum} and
\cite{Wunsch-Zworski:Resolvent}.
\end{enumerate}
These are almost `off the shelf' in terms of modern microlocal analysis, and thus
our approach, from a microlocal perspective, is quite simple. We use these
to show that on the continuation across the boundary of the conformally compact
space we have a Fredholm problem, on a perhaps slightly exotic function space,
which however is (perhaps apart from the complex absorption)
the simplest possible coisotropic function space based on
a Sobolev space, with order dictated by the radial points. Also, we propagate
the estimates along bicharacteristics in different directions depending
on the component $\Sigma_\pm$
of the characteristic set under consideration; correspondingly
the sign of the complex absorbing `potential' will vary with $\Sigma_\pm$, which
is perhaps slightly unusual. However, this is completely parallel to solving the
standard Cauchy, or forward, problem for the wave equation, where one propagates
estimates in {\em opposite} directions relative to the Hamilton vector field in the
two components of the characteristic set.

The complex absorption we use modifies the operator $P_\sigma$ outside
$X_{0,\even}$. However, while $(P_\sigma-\imath Q_\sigma)^{-1}$ depends on $Q_\sigma$,
its behavior on $X_{0,\even}$, and even near $X_{0,\even}$, is independent of this
choice; see the proof of Section~\ref{sec:conf-comp-results}
for a detailed explanation. In particular, although $(P_\sigma-\imath Q_\sigma)^{-1}$
may have resonances
other than those of $\cR(\sigma)$, the resonant states of
these additional resonances are supported
outside $X_{0,\even}$, hence do not affect the singular behavior of the resolvent in
$X_{0,\even}$.

While the results are stated for the scalar equation, analogous results hold
for operators on natural vector bundles, such as the Laplacian on differential
forms. This is so because the results work if the principal symbol of the extended
problem is scalar
with the demanded properties, and the principal symbol of
$\frac{1}{2\imath}(P_\sigma-P_\sigma^*)$ is either scalar at the
`radial sets', or instead satisfies appropriate estimates (as an endomorphism of
the pull-back of the vector bundle to the cotangent bundle) at this location;
see Remark~\ref{rem:bundles}.
The only change in terms of results on asymptotically hyperbolic spaces
is that the threshold $(\dimnm)^2/4$ is shifted; in
terms of the explicit conjugation of Section~\ref{sec:conf-comp-results}
this is so because of the change in the first order term in \eqref{eq:conf-comp-Lap-form}.

In Section~\ref{sec:spaces} we describe in detail the setup of
conformally compact spaces and the extension across the boundary. Then
in Section~\ref{sec:microlocal} we describe the in detail the
necessary microlocal analysis for the extended operator.
Finally, in Section~\ref{sec:conf-comp-results} we translate these
results back to asymptotically hyperbolic spaces.

I am very grateful to Maciej Zworski, Richard Melrose, Semyon Dyatlov,
Gunther Uhlmann, Jared
Wunsch, Rafe Mazzeo,
Kiril Datchev, Colin Guillarmou and Dean Baskin for very helpful
discussions, careful reading of versions of this manuscript as well as
\cite{Vasy-Dyatlov:Microlocal-Kerr} (with special thanks to Semyon
Dyatlov in this regard; Dyatlov noticed an incomplete argument in an
earlier version of this paper), and
for their enthusiasm for this project, as well as to participants in my Topics in
Partial Differential Equations class at Stanford University in Winter Quarter
2011, where this material was covered, for their questions and
comments.

\section{Notation}\label{sec:notation}
We start by briefly recalling the basic pseudodifferential objects, in
part to establish notation.
As a general reference for microlocal analysis, we refer to
\cite{Hor}, while for semiclassical analysis, we refer to
\cite{Dimassi-Sjostrand:Spectral, Evans-Zworski:Semiclassical}.

First,
$S^\diffordgen(\RR^p;\RR^\ell)$ is
the set of $\CI$ functions on $\RR^p_z\times\RR^\ell_\zeta$
satisfying uniform
bounds
$$
|D_z^\alpha D_\zeta^\beta a|\leq C_{\alpha\beta}\langle \zeta\rangle^{\diffordgen-|\beta|},
\ \alpha\in\Nat^p,\ \beta\in\Nat^\ell.
$$
If $O\subset\RR^p$ and $\Gamma\subset\RR^\ell_\zeta$ are open,
we define $S^\diffordgen(O;\Gamma)$ by requiring
these estimates to hold
only for $z\in O$ and $\zeta\in\Gamma$.
(We could instead require uniform estimates on compact subsets;
this makes no difference here.)
The class of classical (or one-step polyhomogeneous) symbols is
the subset $S_{\cl}^\diffordgen(\RR^p;\RR^\ell)$ of $S^\diffordgen(\RR^p;\RR^\ell)$
consisting of symbols possessing an asymptotic expansion
\begin{equation}\label{eq:classical-expand}
a(z,r\omega)\sim \sum a_j(z,\omega) r^{\diffordgen-j},
\end{equation}
where $a_j\in\CI(\RR^p\times\sphere^{\ell-1})$.
Then on $\RR^n_z$, pseudodifferential
operators $A\in\Psi^{\diffordgen}(\RR^n)$ are of the form
\begin{equation*}\begin{split}
A=\Op(a);\ &(\Op(a)u)(z)
=(2\pi)^{-n}\int_{\RR^n} e^{i(z-z')\cdot\zeta} a(z,\zeta)\,u(z')\,d\zeta\,dz',\\
&\qquad u\in\cS(\RR^n),\ a\in S^\diffordgen(\RR^n;\RR^n);
\end{split}\end{equation*}
understood as an oscillatory integral. Classical pseudodifferential operators,
$A\in\Psi_{\cl}^\diffordgen(\RR^n)$, form the subset where $a$ is a classical symbol.
The principal symbol $\sigma_{\diffordgen}(A)$
of $A\in\Psi^\diffordgen(\RR^n)$ is the equivalence class
$[a]$ of $a$ in $S^\diffordgen(\RR^n;\RR^n)/S^{\diffordgen-1}(\RR^n;\RR^n)$.
For classical $a$, one can instead regard $a_0(z,\omega) r^\diffordgen$ as the
principal symbol; it is a $\CI$ function on $\RR^n\times(\RR^n\setminus \{0\})$,
which is homogeneous of degree $\diffordgen$ with respect to the $\RR^+$-action
given by dilations in the second factor, $\RR^n\setminus \{0\}$. The
principal symbol is multiplicative,
i.e.\
$\sigma_{\diffordgen+\diffordgen'}(AB)=\sigma_{\diffordgen}(A)\sigma_{\diffordgen'}(B)$. Moreover,
the principal symbol of a commutator is given by the Poisson bracket
(or equivalently by the Hamilton vector field):
$\sigma_{\diffordgen+\diffordgen'-1}(\imath[A,B])=\sH_{\sigma_{\diffordgen}(A)}\sigma_{\diffordgen'}(B)$,
with
$\sH_a=\sum_{j=1}^n((\pa_{\zeta_j}a)\pa_{z_j}-(\pa_{z_j}a)\pa_{\zeta_j})$. Note
that for $a$ homogeneous of order $\diffordgen$, $\sH_a$ is
homogeneous of order $\diffordgen-1$.

There are two very important properties: non-degeneracy (called ellipticity) and extreme
degeneracy (captured by the operator wave front set) of an operator. One
says that $A$ is elliptic at $\alpha\in \RR^n\times(\RR^n\setminus
\{0\})$ if there exists an open cone $\Gamma$ (conic with respect to
the $\RR^+$-action on
$\RR^n\setminus o$) around $\alpha$ and $R>0$, $C>0$ such that
$|a(x,\xi)|\geq C|\xi|^{\diffordgen}$ for $|\xi|>R$,
$(x,\xi)\in\Gamma$, where $[a]=\sigma_{\diffordgen}(A)$. If $A$ is
classical, and $a$ is taken to be homogeneous, this just amounts to
$a(\alpha)\neq 0$.

On the other hand, for $A=\Op(a)$ and
$\alpha\in\RR^n\times(\RR^n\setminus o)$ one says that
$\alpha\notin\WF'(A)$
if there exists an open cone $\Gamma$ around $\alpha$ such that
$a|_{\Gamma}\in S^{-\infty}(\Gamma)$, i.e.\ $a|_{\Gamma}$ is rapidly
decreasing, with all derivatives, as $|\xi|\to\infty$,
$(x,\xi)\in\Gamma$.
Note that both the elliptic set $\elliptic(A)$ of $A$ (i.e.\ the set of points where
$A$ is elliptic) and $\WF'(A)$ are conic.

Differential operators on $\RR^n$ form the subset of $\Psi(\RR^n)$
in which $a$ is polynomial in the second factor, $\RR^n_\zeta$, so locally
$$
A=\sum_{|\alpha|\leq\diffordgen} a_\alpha(z)D_z^\alpha,\qquad
\sigma_{\diffordgen}(A)=\sum_{|\alpha|=\diffordgen} a_\alpha(z)\zeta^\alpha.
$$

If $X$ is a manifold,
one can transfer these definitions to $X$ by localization and requiring that
the Schwartz kernels are $\CI$ densities away from the diagonal in $X^2=X\times X$;
then $\sigma_{\diffordgen}(A)$ is in $S^\diffordgen(T^*X)/S^{\diffordgen-1}(T^*X)$,
resp.\ $S^\diffordgen_{\hom}(T^*X\setminus o)$ when
$A\in\Psi^\diffordgen(X)$, resp.\ $A\in\Psi^{\diffordgen}_{\cl}(X)$; here $o$ is the zero section,
and $\hom$ stands for symbols homogeneous with respect to the $\RR^+$ action.
If $A$ is a differential operator, then the classical (i.e.\ homogeneous) version
of the principal symbol is a homogeneous polynomial in the fibers of the
cotangent bundle of degree $\diffordgen$. The notions of $\elliptic(A)$ and $\WF'(A)$
extend to give conic subsets of $T^*X\setminus o$; equivalently they
are subsets of the cosphere bundle $S^*X=(T^*X\setminus o)/\RR^+$.
We can also work with operators depending on a parameter $\lambda\in O$ by
replacing $a\in S^\diffordgen(\RR^n;\RR^n)$ by $a\in S^\diffordgen(\RR^n\times O;\RR^n)$,
with $\Op(a_\lambda)\in\Psi^\diffordgen(\RR^n)$ smoothly dependent on $\lambda\in O$.
In the case of differential operators, $a_\alpha$ would simply depend smoothly
on the parameter $\lambda$.

We next consider the semiclassical operator algebra.
We adopt the convention that $\semi$
denotes semiclassical objects, while $h$ is the actual semiclassical
parameter.
This algebra, $\Psih(\RR^n)$, is given by
\begin{equation*}\begin{split}
A_h=\Oph(a);\ &\Oph(a)u(z)=(2\pi h)^{-n}\int_{\RR^n} e^{i(z-z')\cdot\zeta/h} a(z,\zeta,h)\,u(z')\,d\zeta\,dz',\\
&\qquad u\in\cS(\RR^n),\ a\in \CI([0,1)_h;
S^\diffordgen(\RR^n;\RR^n_\zeta));
\end{split}\end{equation*}
its classical subalgebra, $\Psihcl(\RR^n)$ corresponds to
$a\in \CI([0,1)_h;
S_{\cl}^\diffordgen(\RR^n;\RR^n_\zeta))$.
The semiclassical principal symbol is now $\sigma_{\semi,\diffordgen}(A)=a|_{h=0}\in
S^\diffordgen(\RR^n\times\RR^n)$. In the setting of a general manifold $X$,
$\RR^n\times\RR^n$
is replaced by $T^*X$. Correspondingly, $\WFh'(A)$ and
$\elliptic_{\semi}(A)$ are subsets of $T^*X$.
We can again add an extra parameter $\lambda\in O$, so
$a\in \CI([0,1)_h;S^\diffordgen(\RR^n\times O;\RR^n_\zeta))$; then in the invariant setting
the principal symbol is $a|_{h=0}\in
S^\diffordgen(T^*X\times O)$.

Differential operators now take the form
\begin{equation}\label{eq:semicl-diff}
A_{h,\lambda}=\sum_{|\alpha|\leq\diffordgen} a_{\alpha}(z,\lambda;h)(hD_z)^\alpha.
\end{equation}
Such a family has two principal symbols, the standard one (but taking into
account the semiclassical degeneration, i.e.\ based on $(hD_z)^\alpha$ rather than
$D_z^\alpha$), which depends on $h$ and is homogeneous,
and the semiclassical one, which is at $h=0$, and is not homogeneous:
\begin{equation*}\begin{split}
&\sigma_{\diffordgen}(A_{h,\lambda})=\sum_{|\alpha|=\diffordgen} a_{\alpha}(z,\lambda;h)
\zeta^\alpha,\\
&\sigma_{\semi}(A_{h,\lambda})
=\sum_{|\alpha|\leq \diffordgen} a_{\alpha}(z,\lambda;0)\zeta^\alpha.
\end{split}\end{equation*}
However, the restriction of $\sigma_{\diffordgen}(A_{h,\lambda})$ to $h=0$ is the
principal symbol of $\sigma_{\semi}(A_{h,\lambda})$. In the special case in which
$\sigma_{\diffordgen}(A_{h,\lambda})$ is independent of $h$ (which is true in
the setting considered below), one can simply regard the usual principal
symbol as the principal part of the semiclassical symbol.

This is a convenient place to recall from \cite{RBMSpec} that it is
often useful to consider the radial compactification of the fibers of
the cotangent bundle to balls (or hemispheres, in the exposition of
\cite{RBMSpec}). Thus, one adds a sphere at infinity to the fiber
$T_q^*X$ of $T^*X$ over each $q\in X$. This sphere
is naturally identified with $S_q^*X$, and we obtain compact fibers
$\overline{T}_q^*X$
with boundary $S^*_q X$, with the smooth structure
near $S_q^*X$ arising from reciprocal polar coordinates
$(\tilde\rho,\omega)=(r^{-1},\omega)$ for $\tilde\rho>0$, but
extending to $\tilde\rho=0$, and with $S_q^*X$ given by
$\tilde\rho=0$. Thus, with $X=\RR^n$
the classical expansion \eqref{eq:classical-expand} becomes
\begin{equation*}
a(z,\tilde\rho,\omega)\sim \tilde\rho^{-\diffordgen}\sum a_j(z,\omega) \tilde\rho^{j},
\end{equation*}
where $a_j\in\CI(\RR^p\times\sphere^{\ell-1})$, so in particular for
$\diffordgen=0$, this is simply the Taylor series expansion at $S^*X$ of a
function smooth up to $S^*X=\pa\overline{T}^*X$.
In the semiclassical context then one considers
$\overline{T}^*X\times[0,1)$, and notes that `classical' semiclassical
operators of order $0$ are given locally by $\Op_{\semi}(a)$ with $a$
extending to be smooth up to the boundaries of this space, with
semiclassical symbol given by restriction to
$\overline{T}^*X\times\{0\}$, and standard symbol given by restriction
to $S^*X\times[0,1)$. Thus, the claim regarding the limit of the
semiclassical symbol at infinity is simply a matching statement of the
two symbols at the corner $S^*X\times\{0\}$ in
this compactified picture.

Finally, we recall that if
$P=\sum_{|\alpha|\leq\diffordgen}a_\alpha(z)D_z^\alpha$
is an order $\diffordgen$ differential
operator, then the behavior of $P-\lambda$ as $\lambda\to\infty$ can
be converted to a semiclassical problem by considering
$$
P_{\semi,\sigma}=h^\diffordgen(P-\lambda)
=\sum_{|\alpha|\leq\diffordgen}h^{\diffordgen-|\alpha|}a_\alpha(z)(hD_z)^\alpha-\sigma,
$$
where
$\sigma=h^\diffordgen\lambda$. Here there is freedom in choosing $h$,
e.g.\ $h=|\lambda|^{1/\diffordgen}$, in which case $|\sigma|=1$, but it
is often useful to leave some flexibility in the choice so that $h\sim
|\lambda|^{1/\diffordgen}$ only, and thus $\sigma$ is in a compact subset of
$\Cx$ disjoint from $0$. Note that
$$
\sigma_{\semi}(P_{\semi,\sigma})=\sum_{|\alpha|=\diffordgen}a_\alpha(z)\zeta^\alpha-\sigma.
$$
If we do not want to explicitly multiply by $h^\diffordgen$, we write
the full high-energy principal symbol of $P-\lambda$ as
$$
\sigma_{\full}(P_{\lambda})=\sum_{|\alpha|=\diffordgen}a_\alpha(z)\zeta^\alpha-\lambda.
$$

More generally, if
$P(\lambda)=\sum_{|\alpha|+|\beta|\leq\diffordgen}a_\alpha(z)\lambda^\beta D_z^\alpha$
is an order $\diffordgen$ differential
operator depending on a large parameter $\lambda$, we let
$$
\sigma_{\full}(P(\lambda))=\sum_{|\alpha|+|\beta|=\diffordgen}a_\alpha(z)\lambda^\beta\zeta^\alpha
$$
be the full large-parameter symbol. With $\lambda=h^{-1}\sigma$,
$$
P_{\semi,\sigma}=h^{\diffordgen}P(\lambda)=\sum_{|\alpha|+|\beta|\leq\diffordgen}h^{\diffordgen-|\alpha|-|\beta|}a_\alpha(z)\sigma^\beta
(hD_z)^\alpha
$$
is a semiclassical differential operator with semiclassical symbol
$$
\sigma_{\semi}(P_{\semi,\sigma})=\sum_{|\alpha|+|\beta|=\diffordgen}a_\alpha(z)\sigma^\beta\zeta^\alpha.
$$
Note that the full large-parameter symbol and the semiclassical symbol
are `the same', i.e.\ they are simply related to each other.

\section{Conformally compact spaces}\label{sec:spaces}

\subsection{From the Laplacian to the extended operator}\label{subsec:Laplacian-extension}
Suppose that $g_0$ is an even asymptotically hyperbolic metric on
$X_0$, with $\dim X_0=\dimn$. Then
we may choose a product decomposition near the boundary such that
\begin{equation}\label{eq:ah-g-0-prod}
g_0=\frac{dx^2+h}{x^2}
\end{equation}
there,
where $h$ is an even family of metrics; it is convenient to take $x$ to be a
globally defined boundary defining function. Then the dual metric is
$$
G_0=x^2(\pa_x^2+H),
$$
with $H$ the dual metric family of $h$ (depending on $x$ as a
parameter), and
$$
|dg_0|=\sqrt{|\det g_0|}\,dx\,dy=x^{-\dimn}\sqrt{|\det h|}\,dx\,dy
$$
so
\begin{equation}\label{eq:conf-comp-Lap-form}
\Delta_{g_0}=(xD_x)^2+\imath(\dimnm+x^2\gamma) (xD_x)+x^2\Delta_{h},
\end{equation}
with $\gamma$ even, and $\Delta_h$ the $x$-dependent family of
Laplacians of $h$ on $Y$.

We show now that if we change the smooth structure on $X_0$ by
declaring that only even functions of $x$ are smooth, i.e.\
introducing $\mu=x^2$ as the boundary defining function, then after a
suitable conjugation and division by a vanishing factor the resulting
operator smoothly and non-degenerately continues across the boundary,
i.e.\ continues to $\Xext=(-\delta_0,0)_\mu\times Y \sqcup X_{0,\even}$, where
$X_{0,\even}$ is the manifold $X_0$ with the new smooth structure.

First, changing to coordinates $(\mu,y)$, $\mu=x^2$, we
obtain
\begin{equation}\label{eq:Lap-in-mu}
\Delta_{g_0}=4(\mu D_{\mu})^2+2\imath(\dimnm+\mu\gamma) (\mu D_{\mu})+\mu\Delta_{h},
\end{equation}
Now we conjugate by $\mu^{-\imath\sigman/2+\dimnppar/4}$ to obtain
\begin{equation*}\begin{split}
&\mu^{\imath\sigman/2-\dimnppar/4}(\Delta_{g_0}-\frac{(\dimnm)^2}{4}-\sigman^2)\mu^{-\imath\sigman/2+\dimnppar/4}\\
&=4(\mu D_{\mu}-\sigman/2-\imath \dimnppar/4)^2+2\imath(\dimnm+\mu\gamma) (\mu D_{\mu}-\sigman/2-\imath \dimnppar/4)\\
&\qquad\qquad\qquad\qquad+\mu\Delta_{h}-\frac{(\dimnm)^2}{4}-\sigman^2\\
&=4(\mu D_\mu)^2-4\sigman(\mu D_\mu)+\mu\Delta_h
-4\imath (\mu D_\mu) +2\imath\sigman-1\\
&\qquad\qquad\qquad\qquad+2\imath\mu\gamma (\mu D_{\mu}-\sigman/2-\imath \dimnppar/4).
\end{split}\end{equation*}
Next we multiply by $\mu^{-1/2}$ from both sides to obtain
\begin{equation}\begin{split}\label{eq:ah-prefinal-conj-form}
&\mu^{-1/2}\mu^{\imath\sigman/2-\dimnppar/4}(\Delta_{g_0}-\frac{(\dimnm)^2}{4}-\sigman^2)
\mu^{-\imath\sigman/2+\dimnppar/4}\mu^{-1/2}\\
&=4\mu D_\mu^2-\mu^{-1}-4\sigman D_\mu-2\imath\sigman\mu^{-1}+\Delta_h
-4\imath D_\mu+2\mu^{-1}+2\imath\sigman\mu^{-1}-\mu^{-1}\\
&\qquad\qquad+2\imath\gamma (\mu D_{\mu}-\sigman/2-\imath (\dimnm)/4)\\
&=4\mu D_\mu^2-4\sigman D_\mu+\Delta_h
-4\imath D_\mu+2\imath\gamma (\mu D_{\mu}-\sigman/2-\imath (\dimnm)/4).
\end{split}\end{equation}

This operator is in $\Diff^2(X_{0,\even})$, and now it continues
smoothly across the boundary, by extending $h$ and $\gamma$
in an arbitrary smooth manner. This form suffices for analyzing the
problem for $\sigma$ in a compact set, or indeed for $\sigma$ going to
infinity in a strip near the
reals. However, it is convenient to modify it
as we would like the resulting operator to be semiclassically
elliptic when $\sigma$ is away from the reals. We achieve this
via conjugation by a smooth function, with exponent
depending on $\sigma$. The latter would
make no difference even semiclassically in the real regime as it is conjugation
by an elliptic semiclassical FIO. However, in the
non-real regime (where we would like ellipticity) it does matter; the present operator is
not semiclassically elliptic at the zero section.
So finally we conjugate by $(1+\mu)^{\imath\sigman/4}$ to obtain
\begin{equation}\begin{split}\label{eq:ah-final-conj-form}
P_\sigma=4(1+a_1)\mu D_\mu^2-4(1+a_2)\sigman D_\mu&-(1+a_3)\sigman^2+\Delta_h\\
&-4\imath D_\mu+b_1\mu D_{\mu}+b_2\sigman+c_1
\end{split}\end{equation}
with $a_j$ smooth, real, vanishing at $\mu=0$, $b_j$ and $c_1$ smooth. In
fact, we have $a_1\equiv 0$, but it is sometimes convenient to have
more flexibility in the form of the operator since this means that we
do not need to start from the relatively rigid form
\eqref{eq:conf-comp-Lap-form}.

Writing covectors as
$$
\xi\,d\mu+\eta\,dy,
$$ 
the principal symbol of $P_\sigma\in\Diff^2(\Xext)$, including in the
high energy sense ($\sigma\to\infty$), is
\begin{equation}\begin{split}\label{eq:p-sig-symbol-dS}
p_{\full}&=4(1+a_1)\mu\xi^2-4(1+a_2)\sigman\xi-(1+a_3)\sigman^2+|\eta|_{\mu,y}^2,
\end{split}\end{equation}
and is real for $\sigma$ real.
The Hamilton vector field is
\begin{equation}\begin{split}\label{eq:Ham-vf-p-sig-dS}
\sH_{p_\full}&=4(2(1+a_1)\mu\xi-(1+a_2)\sigman)\pa_\mu+\tilde\sH_{|\eta|^2_{\mu,y}}\\
&\qquad-\Big(4(1+a_1+\mu\frac{\pa a_1}{\pa\mu})\xi^2-4\frac{\pa
  a_2}{\pa\mu}\sigman\xi+\frac{\pa a_3}{\pa\mu}\sigman^2
+\frac{\pa|\eta|^2_{\mu,y}}{\pa\mu}\Big)\pa_\xi\\
&\qquad-\Big(4\frac{\pa a_1}{\pa y}\mu\xi^2-4\frac{\pa a_2}{\pa
  y}\sigman\xi
-\frac{\pa a_3}{\pa y}\sigman^2\Big)\pa_\eta,
\end{split}\end{equation}
where $\tilde\sH$ indicates that this is the Hamilton vector field in
$T^*Y$, i.e.\ with $\mu$ considered a parameter.
Correspondingly, the standard, `classical', principal symbol is
\begin{equation}\begin{split}\label{eq:p-sig-symbol-dS-standard}
p=\sigma_2(P_\sigma)
&=4(1+a_1)\mu\xi^2+|\eta|_{\mu,y}^2,
\end{split}\end{equation}
which is real, independent of $\sigma$,
while the Hamilton vector field is
\begin{equation}\begin{split}\label{eq:Ham-vf-p-sig-dS-standard}
\sH_p&=8(1+a_1)\mu\xi\pa_\mu+\tilde\sH_{|\eta|^2_{\mu,y}}\\
&\qquad-\Big(4(1+a_1+\mu\frac{\pa a_1}{\pa\mu})\xi^2
+\frac{\pa|\eta|^2_{\mu,y}}{\pa\mu}\Big)\pa_\xi-4\frac{\pa a_1}{\pa y}\mu\xi^2\pa_\eta.
\end{split}\end{equation}

It is useful to keep in mind that as
$\Delta_{g_0}-\sigma^2-(\dimnm)^2/4$ is formally self-adjoint relative
to the metric density $|dg_0|$ for $\sigma$ real, so the same holds
for $\mu^{-1/2}(\Delta_{g_0}-\sigma^2-(\dimnm)^2/4)\mu^{-1/2}$ (as
$\mu$ is real), and indeed for
its conjugate by $\mu^{-\imath\sigman/2}(1+\mu)^{\imath\sigman/4}$ for
$\sigma$ real since this is merely unitary conjugation. As for $f$
real, $A$ formally self-adjoint relative to $|dg_0|$, $f^{-1}Af$ is
formally self-adjoint relative to $f^2|dg_0|$, we then deduce that for
$\sigman$ real,
$P_\sigma$ is formally self-adjoint relative to
$$
\mu^{(\dimn+1)/2}|dg_0|=\frac{1}{2}|dh|\,|d\mu|,
$$
as
$x^{-\dimn}\,dx=\frac{1}{2}\mu^{-(\dimn+1)/2}\,d\mu$. Note that
$\mu^{(\dimn+1)/2}|dg_0|$
thus extends to a $\CI$ density to $\Xext$, and we deduce that with
respect to the extended density,
$\sigma_1(\frac{1}{2\imath}(P_\sigma-P_\sigma^*))|_{\mu\geq 0}$ vanishes
when $\sigma\in\RR$. Since in general $P_\sigma-P_{\re\sigma}$ differs
from $-4\imath(1+a_2)\im\sigman D_\mu$ by a zeroth order operator, we
conclude that
\begin{equation}\label{eq:subpr-at-rad}
\sigma_1\Big(\frac{1}{2\imath}(P_\sigma-P_\sigma^*)\Big)\Big|_{\mu=0}=-4(\im\sigman)\xi.
\end{equation}

We still need to check that $\mu$ can be appropriately chosen in the interior
away from the region of validity of the product decomposition \eqref{eq:ah-g-0-prod}
(where we had no requirements so far on $\mu$). This only matters for semiclassical
purposes, and (being smooth and non-zero in the interior)
the factor $\mu^{-1/2}$ multiplying from both sides does not affect
any of the relevant properties (semiclassical ellipticity and possible non-trapping
properties), so can be ignored --- the same is true for $\sigma$-independent
powers of $\mu$.

Thus, near $\mu=0$, but $\mu$ bounded away from $0$, the only
semiclassically non-trivial action we have done was to conjugate the
operator by $e^{-\imath\sigma\phi}$ where
$e^{\phi}=\mu^{1/2}(1+\mu)^{-1/4}$; we need to extend $\phi$ into the interior.
But the semiclassical principal symbol of the conjugated operator is,
with $\sigma=z/h$,
\begin{equation}\label{eq:semicl-pr-hyp}
(\zeta-
z\,d\phi,\zeta-z\,d\phi)_{G_0}-z^2=|\zeta|^2_{G_0}-2z(\zeta,d\phi)_{G_0}
-(1-|d\phi|^2_{G_0})z^2.
\end{equation}
For $z$ non-real this is elliptic if $|d\phi|_{G_0}<1$. Indeed, if \eqref{eq:semicl-pr-hyp}
vanishes then from the vanishing imaginary part we get
\begin{equation}\label{eq:semicl-pr-hyp-im}
2\im z ((\zeta,d\phi)_{G_0}+(1-|d\phi|^2_{G_0})\re z)=0,
\end{equation}
and then the real part is
\begin{equation}\begin{split}\label{eq:semicl-pr-hyp-re}
&|\zeta|^2_{G_0}-2\re z (\zeta,d\phi)_{G_0}
-(1-|d\phi|^2_{G_0})((\re z)^2-(\im z)^2)\\
&=|\zeta|^2_{G_0}
+(1-|d\phi|^2_{G_0})((\re z)^2+(\im z)^2),
\end{split}\end{equation}
which cannot vanish if $|d\phi|_{G_0}<1$.
But, reading off the dual metric from the principal symbol of \eqref{eq:Lap-in-mu},
$$
\frac{1}{4}
\left|d(\log\mu-\frac{1}{2}\log(1+\mu))\right|^2_{G_0}=\left(1-\frac{\mu}{2(1+\mu)}\right)^2<1
$$
for $\mu>0$, with a strict bound as long as $\mu$ is bounded away from
$0$. Correspondingly, $\mu^{1/2}(1+\mu)^{-1/4}$ can be extended to a function
$e^\phi$ on all of $X_0$ so that semiclassical ellipticity for $z$
away from the reals is preserved,
and we may even require that $\phi$ is constant on
a fixed (but arbitrarily large) compact subset of $X_0^\circ$. Then,
after conjugation by $e^{-\imath\sigma\phi}$,
\begin{equation}\label{eq:P-semicl-form}
P_{h,z}=e^{\imath z\phi/h}\mu^{-\dimnppar/4-1/2}(h^2\Delta_{g_0}-z)\mu^{\dimnppar/4-1/2} e^{-\imath z\phi/h}
\end{equation}
is
semiclassically elliptic in $\mu>0$ (as well as in $\mu\leq 0$, $\mu$
near $0$, where this is already guaranteed), as desired.

\begin{rem}\label{rem:bundles}
We have not considered vector bundles over $X_0$.
However, for instance for the Laplacian on the differential form bundles it is
straightforward to check that slightly changing the power of $\mu$ in
the conjugation the resulting operator extends smoothly across $\pa
X_0$, has scalar principal symbol of the form
\eqref{eq:p-sig-symbol-dS},
and the principal symbol of $\frac{1}{2\imath}(P_\sigma-P_\sigma^*)$,
which plays a role below, is also as in
the scalar setting, so all the results in fact go through.
\end{rem}

\subsection{Local dynamics near the radial set}
Let
$$
N^*S\setminus o=\Lambda_+\cup\Lambda_-,
\qquad \Lambda_\pm=N^*S\cap\{\pm\xi>0\},\qquad S=\{\mu=0\};
$$
thus $S\subset\Xext$ can be identified with $Y=\pa X_0(=\pa X_{0,\even})$.
Note that $p=0$ at $\Lambda_\pm$ and $\sH_p$ is radial there since
$$
N^*S=\{(\mu,y,\xi,\eta):\ \mu=0,\ \eta=0\},
$$
so
$$
\sH_p|_{N^*S}=-4\xi^2\pa_\xi.
$$
This  corresponds to $dp=4\xi^2\,d\mu$ at $N^*S$, so the
characteristic set $\Sigma=\{p=0\}$ is smooth at $N^*S$.

Let $L_\pm$ be the image of $\Lambda_\pm$ in $S^*\Xext$.
Next we analyze the Hamilton flow at $\Lambda_\pm$. First,
\begin{equation}\label{eq:eta-loc}
\sH_p|\eta|^2_{\mu,y}=8(1+a_1)\mu\xi\pa_\mu|\eta|^2_{\mu,y}-4\frac{\pa a_1}{\pa y}\mu\xi^2\cdot_h\eta
\end{equation}
and
\begin{equation}\label{eq:mu-exact-ev}
\sH_p\mu=8(1+a_1)\xi\mu.
\end{equation}
In terms of linearizing the flow at $N^*S$, $p$ and $\mu$ are
equivalent as $dp=4\xi^2\,d\mu$ there,
so one can simply use $\hat p=p/|\xi|^2$ (which is homogeneous
of degree $0$, like $\mu$), in place of $\mu$.
Finally,
\begin{equation}
\sH_p|\xi|=-4\sgn(\xi)+b,
\end{equation}
with $b$ vanishing at $\Lambda_\pm$.

\begin{figure}[ht]
\begin{center}
\mbox{\epsfig{file=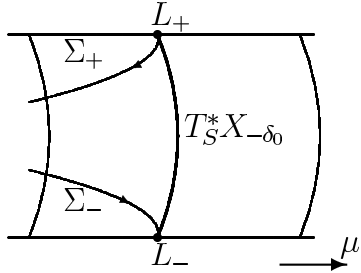}}
\end{center}
\caption{The cotangent bundle of $\Xext$ near $S=\{\mu=0\}$.
It is drawn in a fiber-radially compactified view. The boundary of the fiber
compactificaton is the cosphere
bundle $S^*\Xext$; it is the surface of the cylinder shown.
$\Sigma_\pm$ are the components of the (classical) characteristic set containing
$L_\pm$. They lie in $\mu\leq 0$, only meeting $S^*_S\Xext$ at $L_\pm$.
Semiclassically, i.e.\ in the interior of $\overline{T}^*\Xext$, for $z=h^{-1}\sigma>0$,
only the component of the semiclassical characteristic set containing $L_+$ can
enter $\mu>0$. This is reversed for $z<0$.}
\label{fig:event-horizon-bundle1}
\end{figure}

It is convenient to rehomogenize \eqref{eq:eta-loc} in terms
of $\hat\eta=\eta/|\xi|$. This can be phrased more invariantly by
working with $S^*\Xext=(T^*\Xext\setminus o)/\RR^+$, briefly discussed
in Section~\ref{sec:notation}. Let $L_\pm$ be the image of
$\Lambda_\pm$ in $S^*\Xext$. Homogeneous degree zero
functions on $T^*\Xext\setminus o$, such as $\hat p$,
can be regarded as functions on $S^*\Xext$. For semiclassical purposes, it
is best to consider $S^*\Xext$ as the boundary at fiber infinity of the
fiber-radial compactification $\overline{T}^*\Xext$ of $T^*\Xext$,
also discussed in Section~\ref{sec:notation}.
Then at fiber infinity near $N^*S$, we can take $(|\xi|^{-1},\etah)$
as (projective, rather than polar)
coordinates on the fibers of the cotangent bundle, with $\rhot=|\xi|^{-1}$
defining $S^*\Xext$ in $\overline{T}^*\Xext$.
Then $W=|\xi|^{-1}\sH_p$ is a $\CI$ vector field in this region and
\begin{equation}\label{eq:eta-hat-loc}
|\xi|^{-1}\sH_p|\hat\eta|^2_{\mu,y}
=2|\hat\eta|^2_{\mu,y}\sH_p|\xi|^{-1}+|\xi|^{-3}\sH_p|\eta|^2_{\mu,y}
=8(\sgn\xi) |\hat\eta|^2+\tilde a,
\end{equation}
where $\tilde a$ vanishes cubically at $N^*S$.
In similar notation we have
\begin{equation}\label{eq:weight-definite}
\sH_p\tilde\rho=4\sgn(\xi)+\tilde a',\qquad \tilde\rho=|\xi|^{-1},
\end{equation}
and
\begin{equation}\label{eq:dS-Ham-2}
|\xi|^{-1}\sH_p\mu=8(\sgn\xi)\mu+\tilde a'',
\end{equation}
with $\tilde a'$ smooth (indeed, homogeneous degree zero without the
compactification) vanishing at $N^*S$, and $\tilde a''$ is also
smooth, vanishing quadratically at $N^*S$.
As the vanishing of $\hat\eta,|\xi|^{-1}$ and $\mu$ defines
$\pa N^*S$, we conclude that
$L_-=\pa\Lambda_-$ is a sink, while $L_+=\pa\Lambda_+$ is a source,
in the sense that all nearby
bicharacteristics (in fact, including semiclassical (null)bicharacteristics, since
$\sH_p|\xi|^{-1}$ contains the additional information needed; see
\eqref{eq:Ham-vf-p-h-oT}) converge
to $L_\pm$ as the parameter along the bicharacteristic goes to $\mp\infty$.
In particular, the quadratic defining function of $L_\pm$ given by 
$$
\rho_0=\widehat{\tilde p}+\hat p^2,\ \text{where}\ \hat p=|\xi|^{-2}p,
\ \widehat{\tilde p}=|\etah|^2,
$$
satisfies
\begin{equation}\label{eq:rho-0-property}
(\sgn\xi)W\rho_0\geq 8\rho_0+\cO(\rho_0^{3/2}).
\end{equation}

We also need information on the principal symbol of
$\frac{1}{2\imath}(P_\sigma-P_\sigma^*)$
at the radial points. At $L_\pm$ this is given by
\begin{equation}\label{eq:subpr-symbol-at-Lag}
\sigma_1(\frac{1}{2\imath}(P_\sigma-P_\sigma^*))|_{N^*S}=-(4\sgn(\xi))\im\sigman|\xi|;
\end{equation}
here $(4\sgn(\xi))$ is pulled out due to \eqref{eq:weight-definite}, namely its
size relative to $\sH_p|\xi|^{-1}$ matters.
This corresponds to the fact
that $(\mu\pm \imath 0)^{\imath\sigman}$, which are Lagrangian
distributions
associated to $\Lambda_\pm$, solve the PDE \eqref{eq:ah-final-conj-form}
modulo an error that is two orders lower than what one might a priori expect,
i.e.\ $P_\sigma (\mu\pm \imath 0)^{\imath\sigman}\in (\mu\pm \imath 0)^{\imath\sigman}\CI(\Xext)$.
Note that $P_\sigma$ is second order, so one should lose two orders a
priori, i.e.\ get an element of $(\mu\pm \imath 0)^{\imath\sigman-2}\CI(\Xext)$;
the characteristic nature of $\Lambda_\pm$ reduces the loss to $1$, and the
particular choice of exponent eliminates the loss. This has much in common
with $e^{\imath\lambda/x}x^{(n-1)/2}$ being an approximate solution in asymptotically
Euclidean scattering, see \cite{RBMSpec}.

\subsection{Global behavior of the characteristic set}\label{subsec:global-char}
By \eqref{eq:p-sig-symbol-dS-standard}, points with $\xi=0$ cannot lie
in the characteristic set.
Thus, with
$$
\Sigma_\pm=\Sigma\cap\{\pm\xi>0\},
$$
$\Sigma=\Sigma_+\cup\Sigma_-$
and $\Lambda_\pm\subset\Sigma_\pm$. Further, the characteristic set
lies in $\mu\leq 0$, and intersects $\mu=0$ only in $\Lambda_\pm$.

Moreover, as
$\sH_p\mu=8(1+a_1)\xi\mu$
and $\xi\neq 0$ on $\Sigma$, and $\mu$ only vanishes at
$\Lambda_+\cup\Lambda_-$
there, for $\ep_0>0$ sufficiently small the $\CI$ function
$\mu$ provides a negative global escape function on $\mu\geq -\ep_0$
which is decreasing
on $\Sigma_+$, increasing on $\Sigma_-$. Correspondingly, bicharacteristics
in $\Sigma_-$ travel from $\mu=-\ep_0$ to $L_-$, while in $\Sigma_+$ they travel from
$L_+$ to $\mu=-\ep_0$.

\subsection{High energy, or semiclassical, asymptotics}\label{subsec:global-semi}
We are also interested in the high energy behavior, as $|\sigma|\to\infty$.
For the associated semiclassical problem one obtains a family of operators
$$
P_{h,z}=h^2P_{h^{-1}z},
$$
with $h=|\sigma|^{-1}$, and $z$ corresponding to $\sigma/|\sigma|$ in the unit
circle in $\Cx$.
Then the semiclassical principal symbol $p_{\semi,z}$ of $P_{h,z}$ is a function
on $T^*\Xext$, whose asymptotics at fiber infinity of $T^*\Xext$ is given
by the classical principal symbol $p$.
We are interested in $\im\sigma\geq -C$, which in semiclassical
notation corresponds to $\im z\geq -Ch$.
It is sometimes convenient to think of $p_{\semi,z}$, and its rescaled Hamilton vector field,
as objects on $\overline{T}^*\Xext$.
Thus,
\begin{equation}\begin{split}\label{eq:p-h-symbol-dS}
p_{\semi,z}=\sigma_{2,\semi}(P_{h,z})
&=4(1+a_1)\mu\xi^2-4(1+a_2)\zn\xi-(1+a_3)\zn^2+|\eta|_{\mu,y}^2,
\end{split}\end{equation}
so
\begin{equation}\label{eq:im-p-h-symbol-dS}
\im p_{\semi,z}=-2\im \zn(2(1+a_2)\xi+(1+a_3)\re \zn).
\end{equation}
In particular, for $z$ non-real, $\im p_{\semi,z}=0$ implies
$2(1+a_2)\xi+(1+a_3)\re \zn=0$, so
\begin{equation}\begin{split}\label{eq:re-p-h-symbol-dS}
\re p_{\semi,z}=((1+a_1)(1+a_3)^2(1+a_2)^{-2}\mu&+(1+2a_2)(1+a_3))(\re
\zn)^2\\
&+(1+a_3)(\im \zn)^2+|\eta|^2_{\mu,y}>0
\end{split}\end{equation}
near $\mu=0$,
i.e.\ $p_{\semi,z}$ is semiclassically elliptic on $T^*\Xext$, but {\em not}
at fiber infinity, i.e.\ at $S^*\Xext$ (standard ellipticity is lost only
in $\mu\leq 0$, of course). In $\mu>0$ we have semiclassical
ellipticity (and automatically classical ellipticity) by our choice
of $\phi$ following \eqref{eq:semicl-pr-hyp}.
Explicitly, if we introduce
for instance
\begin{equation}\label{eq:oT-coords}
(\mu,y,\nu,\hat\eta),\qquad \nu=|\xi|^{-1},\ \hat\eta=\eta/|\xi|,
\end{equation}
as valid projective coordinates
in a (large!) neighborhood of $L_\pm$ in $\overline {T}^* \Xext$, then
\begin{equation*}\begin{split}
\nu^{2}p_{\semi,z}=4(1+a_1)\mu-4(1+a_2)(\sgn\xi)\zn\nu-(1+a_3)\zn^2\nu^2+|\hat\eta|_{y,\mu}^2
\end{split}\end{equation*}
so
$$
\nu^{2}\im p_{\semi,z}=-4(1+a_2)(\sgn\xi)\nu\im \zn-2(1+a_3)\nu^2\re \zn\im \zn
$$
which automatically vanishes at $\nu=0$, i.e.\ at $S^*\Xext$. Thus,
for $\sigma$ large and pure imaginary, the semiclassical problem adds no
complexity to the `classical' quantum problem, but of course it does not
simplify it. In fact, we need somewhat more information at the characteristic
set, which is thus at $\nu=0$ when $\im z$ is bounded away from $0$:
\begin{equation*}\begin{split}
&\nu\ \text{small},\ \im z\geq 0\Rightarrow (\sgn\xi)\im p_{\semi,z}\leq 0
\Rightarrow \pm\im p_{\semi,z}\leq 0\ \text{near}\ \Sigma_{\semi,\pm},\\
&\nu\ \text{small},\ \im z\leq 0\Rightarrow (\sgn\xi)\im p_{\semi,z}\geq 0
\Rightarrow \pm\im p_{\semi,z}\geq 0\ \text{near}\ \Sigma_{\semi,\pm},\\
\end{split}
\end{equation*}
which, as we recall in Section~\ref{sec:microlocal},
means that for $P_{h,z}$ with $\im z>0$ one can
propagate estimates forwards along the bicharacteristics where $\xi>0$ (in
particular, away from $L_+$, as the latter is a source) and
backwards where $\xi<0$ (in particular, away from $L_-$,
as the latter is a sink), while for $P^*_{h,z}$ the directions are
reversed since its semiclassical symbol is $\overline{p_{\semi,z}}$. The
directions are also reversed if $\im z$ switches sign. This
is important because it gives invertibility for $z=\imath$ (corresponding
to $\im\sigma$ large positive, i.e.\ the physical halfplane), but does not give
invertibility for $z=-\imath$ negative.

We now return to the claim that even semiclassically, for $z$ almost
real (i.e.\ when $z$ is not bounded away from the reals; we are not
fixing $z$ as we let $h$ vary!), when the operator is not semiclassically elliptic on $T^*\Xext$
as mentioned above,
the characteristic
set can be divided into two components $\Sigma_{\semi,\pm}$,
with $L_\pm$ in different
components.
The vanishing of the factor following $\im z$ in
\eqref{eq:im-p-h-symbol-dS} gives a hypersurface that separates $\Sigma_{\semi}$
into two parts. Indeed,
this is the hypersurface given by
\begin{equation}\label{eq:dS-sep-hyp}
2(1+a_2)\xi+(1+a_3)\re \zn=0,
\end{equation}
on which, by \eqref{eq:re-p-h-symbol-dS}, $\re p_{\semi,z}$ cannot
vanish, so
$$
\Sigma_{\semi}=\Sigma_{\semi,+}\cup\Sigma_{\semi,-},
\qquad \Sigma_{\semi,\pm}=\Sigma_{\semi}\cap\{\pm (2(1+a_2)\xi+(1+a_3)\re \zn)>0\}.
$$
Farther in $\mu>0$, the hypersurface is given, due to \eqref{eq:semicl-pr-hyp-im},
by
$$
(\zeta,d\phi)_{G_0}+(1-|d\phi|^2_{G_0})\re z=0,
$$
and on it, by \eqref{eq:semicl-pr-hyp-re}, the real part is
$|\zeta|^2_{G_0}+(1-|d\phi|^2_{G_0})((\re z)^2+(\im z)^2)>0$;
correspondingly
$$
\Sigma_{\semi}=\Sigma_{\semi,+}\cup\Sigma_{\semi,-},
\qquad 
\Sigma_{\semi,\pm}=\Sigma_{\semi}\cap\{\pm ((\zeta,d\phi)_{G_0}+(1-|d\phi|^2_{G_0})\re z)>0\}.
$$
In fact, more generally, the real part is
\begin{equation*}\begin{split}
&|\zeta|^2_{G_0}-2\re z (\zeta,d\phi)_{G_0}-(1-|d\phi|^2_{G_0})((\re
z)^2-(\im z)^2)\\
&=|\zeta|^2_{G_0}-2\re z ((\zeta,d\phi)_{G_0}+(1-|d\phi|^2_{G_0})\re z)
+(1-|d\phi|^2_{G_0})((\re z)^2+(\im z)^2),
\end{split}\end{equation*}
so for $\pm\re z>0$, $\mp((\zeta,d\phi)_{G_0}+(1-|d\phi|^2_{G_0})\re
z)>0$ implies that $p_{\semi,z}$ does not vanish. Correspondingly,
only one of the two components of $\Sigma_{\semi,\pm}$ enter $\mu>0$,
namely for $\re z>0$, only $\Sigma_{\semi,+}$ enters, while for $\re
z<0$, only $\Sigma_{\semi,-}$ enters.

We finally need more information about the global semiclassical
dynamics.

\begin{lemma}
There exists $\ep_0>0$ such that the following holds.
All semiclassical
null-bicharacteristics in $(\Sigma_{\semi,+}\setminus
L_+)\cap\{-\ep_0\leq\mu\leq\ep_0\}$
go to either
$L_+$ or to $\mu=\ep_0$ in the backward direction and to $\mu=\ep_0$ or $\mu=-\ep_0$
in the forward direction, while all semiclassical
null-bicharacteristics
in $(\Sigma_{\semi,-}\setminus L_-)\cap\{-\ep_0\leq\mu\leq\ep_0\}$
go to $L_-$ or $\mu=\ep_0$ in the forward
direction and to $\mu=\ep_0$ or $\mu=-\ep_0$ in the backward direction.

For $\re z>0$, only
$\Sigma_{\semi,+}$
enters $\mu>0$, so the $\mu=\ep_0$ possibility only applies to
$\Sigma_{\semi,+}$ then, while for $\re z<0$, the analogous remark
applies to $\Sigma_{\semi,-}$.
\end{lemma}

\begin{proof}
We assume that $\re z>0$ for the sake of definiteness.
Observe that
the semiclassical Hamilton vector field is
\begin{equation}\begin{split}\label{eq:Ham-vf-p-h-dS}
\sH_{p_{\semi,z}}&=4(2(1+a_1)\mu\xi-(1+a_2)\zn)\pa_\mu+\tilde\sH_{|\eta|^2_{\mu,y}}\\
&\qquad-\Big(4(1+a_1+\mu\frac{\pa a_1}{\pa\mu})\xi^2-4\frac{\pa
  a_2}{\pa\mu}\zn\xi+\frac{\pa a_3}{\pa\mu}z^2+\frac{\pa|\eta|^2_{\mu,y}}{\pa\mu}\Big)\pa_\xi\\
&\qquad-\Big(4\frac{\pa a_1}{\pa y}\mu\xi^2-4\frac{\pa a_2}{\pa
  y}z\xi
-\frac{\pa a_3}{\pa y}z^2\Big)\pa_\eta;
\end{split}\end{equation}
here we are concerned about $z$ real.
Near $S^*\Xext=\pa\overline{T}^*\Xext$, using the coordinates
\eqref{eq:oT-coords}
(which are valid near the characteristic set)
\begin{equation}\begin{split}\label{eq:Ham-vf-p-h-oT}
W_\semi=\nu\sH_{p_{\semi,z}}&=4(2(1+a_1)\mu(\sgn\xi)-(1+a_2)\zn\nu)\pa_\mu+\nu\tilde\sH_{|\eta|^2_{\mu,y}}\\
&\qquad+(\sgn\xi)\Big(4(1+a_1+\mu\frac{\pa a_1}{\pa\mu})-4\frac{\pa
  a_2}{\pa\mu}\zn(\sgn\xi)\nu+\frac{\pa a_3}{\pa\mu}z^2\nu^2\\
&\qquad\qquad\qquad\qquad+\frac{\pa|\etah|^2_{\mu,y}}{\pa\mu}\Big)(\nu\pa_\nu+\etah\pa_{\etah})\\
&\qquad-\Big(4\frac{\pa a_1}{\pa y}\mu-4(\sgn\xi)\frac{\pa a_2}{\pa
  y}z\nu
-\frac{\pa a_3}{\pa y}z^2\nu^2\Big)\pa_{\etah},
\end{split}\end{equation}
with $\nu\tilde\sH_{|\eta|^2_{\mu,y}}=\sum_{ij}
H_{ij}\etah_i\pa_{y_j}-\sum_{ijk} \frac{\pa H_{ij}}{\pa
  y_k}\etah_i\etah_j\pa_{\etah_k}$ smooth. Thus, $W_\semi$ is a smooth
vector field on the compactified cotangent bundle,
$\overline{T}^*\Xext$ which is tangent to its boundary, $S^*\Xext$,
and $W_\semi-W=\nu W^\sharp$ (with $W$ considered as a homogeneous degree zero
vector field) with $W^\sharp$ smooth and tangent to $S^*\Xext$. In particular,
by \eqref{eq:weight-definite} and \eqref{eq:rho-0-property}, using
that $\rhot^2+\rho_0$ is a quadratic defining function of $L_\pm$,
$$
(\sgn\xi) W_\semi(\rhot^2+\rho_0)\geq
8(\rhot^2+\rho_0)-\cO((\rhot^2+\rho_0)^{3/2})
$$
shows that
there is $\ep_1>0$ such that in $\rhot^2+\rho_0\leq\ep_1$, $\xi>0$,
$\rhot^2+\rho_0$ is strictly increasing along the Hamilton flow except
at $L_+$, while in $\rhot^2+\rho_0\leq\ep_1$, $\xi<0$,
$\rhot^2+\rho_0$ is strictly decreasing along the Hamilton flow except
at $L_-$. Indeed, all null-bicharacteristics in this neighborhood of
$L_\pm$ except the constant ones at $L_\pm$ tend to $L_\pm$ in one
direction and to $\rhot^2+\rho_0=\ep_1$ in the other direction.

Choosing $\ep'_0>0$ sufficiently small, the characteristic set in
$\overline{T}^*\Xext
\cap \{-\ep'_0\leq\mu\leq\ep'_0\}$ is disjoint from
$S^*\Xext\setminus \{\rhot^2+\rho_0\leq\ep_1\}$, and indeed only
contains points in $\Sigma_{\semi,+}$ as $\re z>0$.
Since
$\sH_{p_{\semi,z}}\mu=4(2(1+a_1)\mu\xi-(1+a_2)\zn) $,
it is negative on $\overline{T}^*_{\{\mu=0\}}\Xext\setminus S^*\Xext$.
In
particular, there is a neighborhood $U$ of $\mu=0$ in
$\Sigma_{\semi,+}\setminus S^*\Xext$ on
which the same sign is preserved; since the characteristic set in
$\overline{T}^*\Xext\setminus \{\rhot^2+\rho_0<\ep_1\}$ is compact,
and is indeed a subset of $T^*\Xext\setminus
\{\rhot^2+\rho_0<\ep_1\}$, we deduce that $|\mu|$ is bounded below
on $\Sigma\setminus(U\cup \{\rhot^2+\rho_0<\ep_1\})$, say
$|\mu|\geq\ep''_0>0$ there,
so with $\ep_0=\min(\ep_0',\ep_0'')$,
$\sH_{p_{\semi,z}}\mu<0$ on $\Sigma_{\semi,+}\cap\{-\ep_0\leq\mu\leq \ep_0\}\setminus \{\rhot^2+\rho_0^2<\ep_1\}$.
As $\sH_{p_{\semi,z}}\mu<0$ at
$\mu=0$, bicharacteristics can only cross $\mu=0$ in the outward
direction.

Thus, if $\gamma$ is a bicharacteristic in $\Sigma_{\semi,+}$, there
are two possibilities. If $\gamma$
is disjoint from $\{\rhot^2+\rho_0<\ep_1\}$, it has to go to
$\mu=\ep_0$ in the backward direction and to $\mu=-\ep_0$ in the
forward direction. If $\gamma$ has a point in
$\{\rhot^2+\rho_0<\ep_1\}$, then it has to go to $L_+$ in the
backward direction and to $\rhot^2+\rho_0=\ep_1$ in the forward
direction; if $|\mu|\geq \ep_0$ by the time $\rhot^2+\rho_0=\ep_1$ is
reached,
the result is proved, and otherwise $\sH_{p_{\semi,z}}\mu<0$ in
$\rhot^2+\rho_0\geq\ep_1$,
$|\mu|\leq\ep_0$, shows that the bicharacteristic goes to $\mu=-\ep_0$
in the forward direction.

If $\gamma$ is a bicharacteristic in $\Sigma_{\semi,-}$, only the
second possibility exists, and the bicharacteristic cannot leave
$\{\rhot^2+\rho_0<\ep_1\}$
in $|\mu|\leq\ep_0$, so it reaches $\mu=-\ep_0$ in the backward
direction (as the characteristic set is in $\mu\leq 0$).
\end{proof}

If we assume that $g_0$ is a non-trapping metric, i.e.\
bicharacteristics of $g_0$ in $T^*X_0^\circ\setminus o$ tend to $\pa
X_0$ in both the forward and the backward directions, then $\mu=\ep_0$
can be excluded from the statement of the lemma, and the above
argument gives the following stronger conclusion:
for sufficiently small $\ep_0>0$, and for $\re z>0$,
any bicharacteristic in $\Sigma_{\semi,+}$ in $-\ep_0\leq \mu$ has to go to
$L_+$ in the backward direction, and to $\mu=-\ep_0$
in the forward direction (with the exception of the constant
bicharacteristics at $L_+$), while in $\Sigma_{\semi,-}$, all
bicharacteristics in $-\ep_0\leq\mu$ lie in $-\ep_0\leq
\mu\leq 0$, and go to $L_-$ in the forward direction and to
$\mu=-\ep_0$ in the backward direction (with the exception of the constant
bicharacteristics at $L_-$).

In fact, for applications, it is also useful to remark that for
sufficiently small $\ep_0>0$, and for $\alpha\in T^*X_0$,
\begin{equation}\label{eq:mu-convex-ah}
0<\mu(\alpha)<\ep_0,\ p_{\semi,z}(\alpha)=0\Mand
(\sH_{p_{\semi,z}}\mu)(\alpha)=0\Rightarrow (\sH_{p_{\semi,z}}^2\mu)(\alpha)<0.
\end{equation}
Indeed,
as $\sH_{p_{\semi,z}}\mu=4(2(1+a_1)\mu\xi-(1+a_2)\zn)$, the hypotheses imply
$\zn=2(1+a_1)(1+a_2)^{-1}\mu\xi$ and
\begin{equation*}\begin{split}
0&=p_{\semi,z}\\
&=4(1+a_1)\mu\xi^2-8(1+a_1)\mu\xi^2-4(1+a_1)^2(1+a_2)^{-2}(1+a_3)\mu^2\xi^2
+|\eta|_{\mu,y}^2\\
&=-4(1+a_1)\mu\xi^2-4(1+a_1)^2(1+a_2)^{-2}(1+a_3)\mu^2\xi^2
+|\eta|_{\mu,y}^2,
\end{split}\end{equation*}
so $|\eta|^2_{\mu,y}=4(1+b)\mu\xi^2$, with $b$ vanishing at
$\mu=0$. Thus, at points where $\sH_{p_{\semi,z}}\mu$ vanishes,
writing $a_j=\mu\tilde a_j$,
\begin{equation}\label{eq:mu-convex-ah-2}
\sH_{p_{\semi,z}}^2\mu=8(1+a_1)\mu\sH_{p_{\semi,z}}\xi+8\mu^2\xi
\sH_{p_{\semi,z}}\tilde a_1-4\zn\mu
\sH_{p_{\semi,z}}\tilde a_2=8(1+a_1)\mu\sH_{p_{\semi,z}}\xi+\cO(\mu^2\xi^2).
\end{equation}
Now
\begin{equation*}\begin{split}
\sH_{p_{\semi,z}}\xi&=-(4(1+a_1+\mu\frac{\pa a_1}{\pa\mu})\xi^2-4\frac{\pa
  a_2}{\pa\mu}\zn\xi+\frac{\pa a_3}{\pa\mu}z^2+\frac{\pa|\eta|^2_{\mu,y}}{\pa\mu}).
\end{split}\end{equation*}
Since $\zn\xi$ is $\cO(\mu\xi^2)$ due to $\sH_{p_{\semi,z}}\mu=0$,
$\zn^2$ is $\cO(\mu^2\xi^2)$ for the same reason, and
$|\eta|^2$ and $\pa_\mu|\eta|^2$ are $\cO(\mu\xi^2)$ due to
$p_{\semi,z}=0$,
we deduce that $\sH_{p_{\semi,z}}\xi<0$ for sufficiently small
$|\mu|$, so
\eqref{eq:mu-convex-ah-2} implies \eqref{eq:mu-convex-ah}.
Thus, $\mu$ can be used for gluing
constructions as in \cite{Datchev-Vasy:Gluing-prop}.

\subsection{Complex absorption}\label{subsec:complex-absorb-dS}
The final step of fitting $P_\sigma$ into our general microlocal framework is
moving the problem to a compact manifold, and adding a complex absorbing
second order operator. We thus consider a compact manifold without boundary
$X$ for which $X_{\mu_0}=\{\mu>\mu_0\}$, $\mu_0=-\ep_0<0$, with
$\ep_0>0$ as above, is identified as an open subset with smooth boundary;
it is convenient to take $X$ to be the double
of $X_{\mu_0}$, so there are two copies of $X_{0,\even}$ in $X$.

In the case of hyperbolic space, this doubling process can be realized
from the perspective of $(n+1)$-dimensional Minkowski space. Then, as
mentioned in the introduction, the Poincar\'e model shows up in two
copies, namely in the interior of the future and past light cone
inside the sphere at infinity, while de Sitter space as the
`equatorial belt', i.e.\ the exterior of the light cone at the sphere
at infinity. One can take the Minkowski equatorial plane, $t=0$, as
$\mu=\mu_0$, and place the complex absorption there, thereby
decoupling the future and past hemispheres. See
\cite{Vasy-Dyatlov:Microlocal-Kerr} for more detail.

It is convenient to separate the `classical' (i.e.\ quantum!) and `semiclassical'
problems, for in the former setting trapping for $g_0$ does not matter, while in the
latter it does.

\begin{figure}[ht]
\begin{center}
\mbox{\epsfig{file=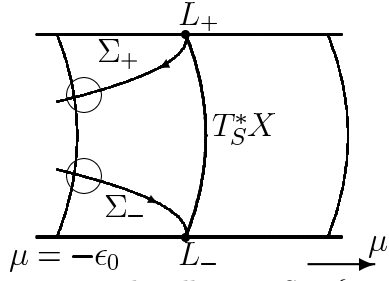}}
\end{center}
\caption{The cotangent bundle near $S=\{\mu=0\}$.
It is drawn in a fiber-radially compactified view, as in Figure~\ref{fig:event-horizon-bundle1}. The circles on the left show the support of $q$; it has opposite signs on
the two disks corresponding to the opposite directions of propagation relative
to the Hamilton vector field.}
\label{fig:event-horizon-damp1}
\end{figure}

We then introduce a `complex
absorption' operator $Q_\sigma\in\Psi_{\cl}^2(X)$
with real principal symbol $q$ supported in, say, $\mu<-\ep_1$, with
the Schwartz kernel also supported in the corresponding region (i.e.\
in both factors on the product space this condition holds on the support)
such that $p\pm\imath q$
is elliptic near $\pa X_{\mu_0}$, i.e.\ near $\mu=\mu_0$,
and which satisfies that $\pm q\geq 0$ near $\Sigma_{\pm}$. This can
easily be done since $\Sigma_\pm$ are disjoint, and away from these
$p$ is elliptic, hence so is $p\pm \imath q$ regardless of the choice
of $q$; we simply need to make $q$ to have support sufficiently close
to $\Sigma_\pm$, elliptic on $\Sigma_\pm$ at $\mu=-\ep_0$,
with the appropriate sign near $\Sigma_\pm$.
Having done this, we extend $p$ and $q$ to $X$ in such a way that
$p\pm\imath q$ are elliptic near $\pa X_{\mu_0}$; the region
we added is thus irrelevant at the level of bicharacteristic dynamics
(of $p$) in so far as it is decoupled from the dynamics in $X_0$,
and indeed also for analysis as we see shortly (in so far as we have
two essentially decoupled copies of the same problem). This is accomplished,
for instance, by using the doubling construction to define $p$ on
$X\setminus X_{\mu_0}$ (in a smooth fashion at $\pa X_{\mu_0}$,
as can be easily arranged; the holomorphic dependence of $P_\sigma$ on
$\sigma$ is still easily preserved), and
then, noting that the characteristic set of $p$ still has two
connected components, making $q$ elliptic on the characteristic set of
$p$ near $\pa X_{\mu_0}$, with the same sign in each component as near
$\pa X_{\mu_0}$.
(An alternative would be to make $q$ elliptic on the characteristic set of
$p$ near $X\setminus X_{\mu_0}$; it is just slightly more complicated
to write down such a $q$ when the high energy behavior is taken into
account.
With the present choice, due to the
doubling, there are essentially two copies of the problem on $X_0$:
the original, and the one from the doubling.)
Finally we take $Q_\sigma$ be
any operator with principal symbol $q$ with Schwartz kernel satisfying
the desired support conditions and which depends on $\sigma$
holomorphically.
We may choose $Q_\sigma$ to be independent of $\sigma$ so $Q_\sigma$
is indeed holomorphic; in this case we may further replace it by
$\frac{1}{2}( Q_\sigma+Q_\sigma^*)$ if self-adjointness is desired.

In view of Subsection~\ref{subsec:global-char} we have arranged the following.
For $\alpha\in S^*X\cap\Sigma$,
let $\gamma_+(\alpha)$, resp.\ $\gamma_-(\alpha)$
denote the image of the forward, resp.\ backward,
half-bicharacteristic of $p$
from $\alpha$. We write $\gamma_\pm(\alpha)\to L_\pm$
(and say $\gamma_\pm(\alpha)$ tends to $L_\pm$) if given any neighborhood $O$
of $L_\pm$, $\gamma_\pm(\alpha)\cap O\neq\emptyset$; by the source/sink property
this implies that the points on the curve are in $O$ for sufficiently large (in absolute
value) parameter values. Then, with $\elliptic(Q_\sigma)$
denoting the elliptic set of $Q_\sigma$,
\begin{multline}\label{eq:classical-non-trap}
\alpha\in\Sigma_-\setminus L_-\Rightarrow \gamma_+(\alpha)\to L_-\Mand
\gamma_-(\alpha)\cap\elliptic(Q_\sigma)\neq\emptyset,\\
\alpha\in\Sigma_+\setminus L_+\Rightarrow \gamma_-(\alpha)\to L_+\Mand
\gamma_+(\alpha)\cap\elliptic(Q_\sigma)\neq\emptyset.\\
\end{multline}
That is, all forward and backward half-(null)bicharacteristics of
$P_\sigma$ either enter the elliptic set of $Q_\sigma$, or go to $\Lambda_\pm$,
i.e.\ $L_\pm$ in $S^*X$.
The point of the arrangements
regarding $Q_\sigma$ and the flow is that we are able to propagate estimates
forward near where $q\geq 0$, backward near where $q\leq 0$, so by our hypotheses
we can always propagate estimates for $P_\sigma-\imath Q_\sigma$
from $\Lambda_\pm$ towards the elliptic set of
$Q_\sigma$.
On the
other hand, for $P_\sigma^*+\imath Q_\sigma^*$, we can propagate estimates
from the elliptic set of $Q_\sigma$ towards $\Lambda_\pm$.
This behavior
of $P_\sigma-\imath Q_\sigma$ vs.\ $P_\sigma^*+\imath Q_\sigma^*$ is important
for duality reasons.

An alternative to the complex absorption would be simply adding a boundary at $\mu=\mu_0$;
this is easy to do since this is a space-like hypersurface, but this is
slightly unpleasant from the point of view of microlocal analysis as one has
to work on a manifold with boundary (though as mentioned this is easily done,
see \cite{Vasy-Dyatlov:Microlocal-Kerr}).

For the semiclassical problem, when $z$ is almost real (namely
when $\im z$ is bounded away from $0$ we only need to make sure we do
not mess up the semiclassical ellipticity in $T^*\Xext$)
we need to increase the requirements on $Q_\sigma$, and what we need
to do depends on whether $g_0$ is non-trapping.

If $g_0$ is non-trapping, we choose $Q_\sigma$
such that $h^2Q_{h^{-1}z}\in\Psihcl^2(X)$ with
semiclassical principal symbol $q_{\semi,z}$, and in addition to the
above requirement for the classical symbol, we need
semiclassical ellipticity
near $\mu=\mu_0$, i.e.\ that
$p_{\semi,z}-\imath q_{\semi,z}$ and its complex conjugate
are elliptic near $\pa X_{\mu_0}$, i.e.\ near $\mu=\mu_0$,
and which satisfies that for $z$ real $\pm q_{\semi,z}\geq 0$ on $\Sigma_{\semi,\pm}$.
Again, we extend $P_\sigma$ and $Q_\sigma$ to $X$ in such a way that
$p-\imath q$
and $p_{\semi,z}-\imath q_{\semi,z}$ (and thus their complex
conjugates)
are elliptic near $\pa X_{\mu_0}$; the region
we added is thus irrelevant. This is straightforward to arrange if one
ignores that one wants $Q_\sigma$ to be holomorphic: one easily constructs a
function $q_{\semi,z}$ on $T^*X$ (taking into account the disjointness of
$\Sigma_{\semi,\pm}$), and defines $Q_{h^{-1}z}$ to be $h^{-2}$
times the semiclassical quantization of $q_{\semi,z}$ (or any other
operator with the same semiclassical and standard principal symbols). Indeed, for our
purposes this would suffice since we want high energy estimates for
the analytic continuation resolvent on the original space $X_0$ (which
we will know exists by the non-semiclassical argument), and as we shall see, the
resolvent is given by the same formula in terms of $(P_\sigma-\imath
Q_\sigma)^{-1}$ independently whether $Q_\sigma$ is holomorphic in
$\sigma$ (as long as it satisfies the other properties), so there is
no need to ensure the holomorphy of $Q_\sigma$. However, it is
instructive to have an example of a holomorphic family $Q_\sigma$ in a
strip at least: in view of \eqref{eq:im-p-h-symbol-dS} we can take
(with $C>0$)
$$ 
q_{h,z}=2(2(1+a_2)\xi+(1+a_3)\zn)(\xi^2+|\eta|^2+z^2+C^2h^2)^{1/2}\chi(\mu),
$$ 
where $\chi\geq 0$ is supported near $\mu_0$; the corresponding full
symbol is
$$
\sigma_{\full}(Q_\sigma)=2(2(1+a_2)\xi+(1+a_3)\sigma)(\xi^2+|\eta|^2+\sigma^2)^{1/2}\chi(\mu),
$$
and $Q_\sigma$ is taken as a quantization of this full symbol.
Here the square root is defined on $\Cx\setminus[0,-\infty)$, with
real part of the result being positive, and
correspondingly
$q_{h,z}$ is defined away from $h^{-1}z\in\pm \imath
[C,+\infty)$. Note that $\xi^2+|\eta|^2+\sigma^2$ is an elliptic
symbol in $(\xi,\eta,\re\sigma,\im\sigma)$ as long as
$|\im\sigma|<C'|\re\sigma|$, so the corresponding statement also holds
for its square root.
While $q_{h,z}$ is only holomorphic away from $h^{-1}z\in\pm \imath [C,+\infty)$, the
full (and indeed the semiclassical and standard principal) symbols are
actually holomorphic in cones near infinity, and indeed e.g.\ via convolutions
by the Fourier transform of a compactly supported function can be
extended to be holomorphic in $\Cx$, but this is of no importance here.

If $g_0$ is trapping, we need to add complex absorption inside $X_0$
as well, at $\mu=\ep_0$, so we relax the requirement that $Q_\sigma$
is supported in $\mu<-\ep_0/2$ to support in $|\mu|>\ep_0/2$,
but we require in addition to the other
classical requirements that $p_{\semi,z}-\imath q_{\semi,z}$ and its complex conjugate are
elliptic near $\mu=\pm\ep_0$,
and which satisfies that $\pm q_{\semi,z}\geq 0$ on
$\Sigma_{\semi,\pm}$. This can be achieved as above for $\mu$ near
$\mu_0$.
Again, we extend $P_\sigma$ and $Q_\sigma$ to $X$ in such a way that
$p-\imath q$
and $p_{\semi,z}-\imath q_{\semi,z}$ (and thus their complex
conjugates) are elliptic near $\pa X_{\mu_0}$.

In either of these semiclassical cases we have arranged that for
sufficiently small $\delta_0>0$,
$p_{\semi,z}-\imath q_{\semi,z}$ and its complex conjugate are {\em semiclassically
  non-trapping} for $|\im z|<\delta_0$, namely
the bicharacteristics from any point in $\Sigma_{\semi}\setminus (L_+\cup L_-)$
flow to $\elliptic(q_{\semi,z})\cup L_-$ (i.e.\ either enter $\elliptic(q_{\semi,z})$ at some finite
time, or tend to $L_-$) in the forward direction, and to $\elliptic(q_{\semi,z})\cup L_+$
in the backward direction. Here $\delta_0>0$ arises from the
particularly simple choice of $q_{\semi,z}$ for which semiclassical
ellipticity is easy to check for $\im z>0$ (bounded away from $0$) and
small; a more careful analysis would give a specific value of
$\delta_0$, and a more careful choice of $q_{\semi,z}$ would give a
better result.

\section{Microlocal analysis}\label{sec:microlocal}

\subsection{Elliptic and microhyperbolic points}
First, recall the basic elliptic and microhyperbolic regularity
results. Let $\WF^s(u)$ denote the $H^s$ wave front set of a
distribution
$u\in\dist(X)$, i.e.\ $\alpha\notin\WF^s(u)$ if there exists
$A\in\Psi^0(X)$ elliptic at $\alpha$ such that $Au\in H^s$. Elliptic
regularity states that
$$
P_\sigma-\imath Q_\sigma\ \text{elliptic at}\ \alpha,\
\alpha\notin\WF^{s-2}((P_\sigma-\imath Q_\sigma)u)\Rightarrow
\alpha\notin\WF^s(u).
$$
In particular, if $(P_\sigma-\imath Q_\sigma)u\in H^{s-2}$ and
$p-\imath q$ is elliptic at $\alpha$ then $\alpha\notin\WF^s(u)$.
Analogous conclusions apply to $P_\sigma^*+\imath Q_\sigma^*$; since
both $p$ and $q$ are real, $p-\imath q$ is elliptic if and only if
$p+\imath q$ is.

We also have real principal type propagation, in the usual form valid
outside $\supp q$:
$$
\WF^s(u)\setminus(\WF^{s-1}((P_\sigma-\imath Q_\sigma)u)\cup\supp q)
$$
is a union of maximally extended bicharacteristics of $\sH_p$ in the
characteristic set $\Sigma=\{p=0\}$ of $P_\sigma$. Putting it differently,
$$
\alpha\notin\WF^s(u)\cup \WF^{s-1}((P_\sigma-\imath Q_\sigma)u)\cup\supp
q\Rightarrow \tilde\gamma(\alpha)\cap\WF^s(u)=\emptyset,
$$
where $\tilde\gamma(\alpha)$ is the component of the bicharacteristic $\gamma(\alpha)$
of $p$ in the complement of $\WF^{s-1}((P_\sigma-\imath Q_\sigma)u)\cup\supp
q$. If $(P_\sigma-\imath Q_\sigma)u\in H^{s-1}$, then $\WF^{s-1}((P_\sigma-\imath
Q_\sigma)u)=\emptyset$ can be dropped from all statements above; if
$q=0$ one can thus replace $\tilde\gamma$ by $\gamma$.

In general, the result does not hold for non-zero $q$. However, it
holds in one direction (backward/forward) of propagation along $\sH_p$
if $q$ has the correct sign. Thus, let $\tilde\gamma_\pm(\alpha)$ be a
forward ($+$) or backward ($-$) bicharacteristic from $\alpha$,
defined on an interval $I$. If $\pm q\geq 0$ on a neighborhood of
$\tilde\gamma_\pm(\alpha)$ (i.e.\ $q\geq 0$ on a neighborhood of
$\tilde\gamma_+(\alpha)$, or $q\leq 0$ on a neighborhood of
$\tilde\gamma_-(\alpha)$) then (for the corresponding sign)
$$
\alpha\notin\WF^s(u)\Mand  \WF^{s-1}((P_\sigma-\imath Q_\sigma)u)\cap\tilde\gamma_\pm(\alpha)=\emptyset\Rightarrow \tilde\gamma_\pm(\alpha)\cap\WF^s(u)=\emptyset,
$$
i.e.\ one can propagate regularity forward if $q\geq 0$, backward if
$q\leq 0$. A proof of this claim that is completely analogous to
H\"ormander's positive commutator proof in the real principal type
setting can easily be given; see \cite{Nonnenmacher-Zworski:Quantum}
and \cite{Datchev-Vasy:Gluing-prop} in the semiclassical setting;
the changes are minor in the `classical' setting. Note that at points
where $q\neq 0$, just $\alpha\notin \WF^{s-1}((P_\sigma-\imath
Q_\sigma)u)$ implies $\alpha\notin\WF^{s+1}(u)$ (stronger than stated
above), but at points with $q=0$ such an elliptic estimate is
unavailable (unless $P_\sigma$ is elliptic).

As $P_\sigma^*+\imath Q_\sigma^*$ has symbol $p+\imath q$, one can
propagate regularity in the opposite direction as compared to
$P_\sigma-\imath Q_\sigma$.
Thus, if $\mp q\geq 0$ on a neighborhood of
$\tilde\gamma_\pm(\alpha)$ (i.e.\ $q\leq 0$ on a neighborhood of
$\tilde\gamma_+(\alpha)$, or $q\geq 0$ on a neighborhood of
$\tilde\gamma_-(\alpha)$) then (for the corresponding sign)
$$
\alpha\notin\WF^s(u)\Mand  \WF^{s-1}((P_\sigma^*+\imath Q_\sigma^*)u)\cap\tilde\gamma_\pm(\alpha)=\emptyset\Rightarrow \tilde\gamma_\pm(\alpha)\cap\WF^s(u)=\emptyset.
$$

\subsection{Analysis near $\Lambda_\pm$}

The last ingredient in the classical setting is an analogue of Melrose's regularity
result at radial sets which have the same features as ours. Although
it is not stated in this generality in Melrose's paper \cite{RBMSpec},
the proof is easily adapted. Thus, the results are:

At $\Lambda_\pm$, for $s\geq m>(\diffordm-\im\sigma)/2$, we can propagate estimates
{\em away} from $\Lambda_\pm$:

\begin{prop}\label{prop:micro-out}
Suppose $s\geq m>(\diffordm-\im\sigma)/2$, and
$\WF^m(u)\cap \Lambda_\pm=\emptyset$. Then
$$
\Lambda_\pm\cap\WF^{s-1}(P_\sigma u)=\emptyset\Rightarrow \Lambda_\pm\cap\WF^s(u)=\emptyset.
$$
\end{prop}

This is completely analogous to Melrose's estimates in asymptotically Euclidean
scattering theory at the radial sets \cite[Section~9]{RBMSpec}.
Note that the $H^s$ regularity of $u$ at $\Lambda_\pm$
is `free' in the
sense that we do not need to impose $H^s$ assumptions on $u$ anywhere; merely
$H^m$ at $\Lambda_\pm$ does the job; of course, on $P_\sigma u$ one must make the
$H^{s-\diffordm}$ assumption, i.e.\ the loss of one derivative compared to the elliptic setting.
At the cost of changing regularity, one can
propagate estimate {\em towards} $\Lambda_\pm$. Keeping in mind that
taking $P_\sigma^*$ in place of $P_\sigma$, principal symbol of
$\frac{1}{2\imath}(P_\sigma-P_\sigma^*)$
switches sign,
we have the following:

\begin{prop}\label{prop:micro-in}
For $s<(\diffordm+\im\sigma)/2$, and $O$ a neighborhood of
$\Lambda_\pm$,
$$
\WF^s(u)\cap (O\setminus\Lambda_\pm)=\emptyset,\ \WF^{s-1}(P_\sigma^*
u)\cap\Lambda_\pm=\emptyset\Rightarrow \WF^s(u)\cap\Lambda_\pm=\emptyset.
$$
\end{prop}

\begin{proof}[Proof of Propositions~\ref{prop:micro-out}-\ref{prop:micro-in}.]
The proof is a positive commutator estimate. Consider
commutants $C_\ep^*C_\ep$ with $C_\ep\in\Psi^{s-\diffordmpar/2-\delta}(X)$ for $\ep>0$,
uniformly bounded in $\Psi^{s-\diffordmpar/2}(X)$ as $\ep\to 0$; with the
$\ep$-dependence used to regularize the argument. More precisely, let
$$
c=\phi(\rho_0)\tilde\rho^{-s+\diffordmpar/2},\qquad c_\ep=c(1+\ep \tilde\rho^{-1})^{-\delta},
$$
where $\phi\in\CI_c(\RR)$ is identically $1$ near $0$, $\phi'\leq 0$ and
$\phi$ is supported sufficiently close to $0$ so that
\begin{equation}\label{eq:rad-localizer-est}
\rho_0\in\supp d\phi\Rightarrow \pm\tilde\rho\,\sH_p\rho_0>0;
\end{equation}
such $\phi$ exists by \eqref{eq:rho-0-property}. To avoid using the
sharp
G{\aa}rding inequality,
we choose $\phi$ so that $\sqrt{-\phi\phi'}$
is $\CI$.
Note that the sign of $\sH_p \tilde\rho^{-s+\diffordmpar/2}$ depends on the sign of
$-s+\diffordmpar/2$ which explains the difference between $s>\diffordmpar/2$
and $s<\diffordmpar/2$
in Propositions~\ref{prop:micro-out}-\ref{prop:micro-in} when
there are no other contributions to the threshold value of $s$. The contribution of
the principal symbol of $\frac{1}{2\imath}(P_\sigma-P_\sigma^*)$,
however, shifts the critical value $\diffordmpar/2$.

Now let $C\in\Psi^{s-\diffordmpar/2}(X)$
have principal symbol $c$, and have $\WF'(C)\subset \supp \phi\circ\rho_0$, and
let $C_\ep=C S_\ep$, $S_\ep\in\Psi^{-\delta}(X)$ uniformly bounded in
$\Psi^0(X)$ for $\ep>0$, converging
to $\Id$ in $\Psi^{\delta'}(X)$ for $\delta'>0$ as $\ep\to 0$, with principal symbol
$(1+\ep \tilde\rho^{-1})^{-\delta}$. Thus, the principal symbol of $C_\ep$ is $c_\ep$.

First, consider Proposition~\ref{prop:micro-out}.
Then
\begin{equation*}\begin{split}
&\sigma_{2s}(\imath(P^*_\sigma C_\ep^*C_\ep-C_\ep^*C_\ep P_\sigma))
=\sigma_1(\imath(P^*_\sigma-P_\sigma)) c_\ep^2
+2c_\ep\sH_p c_\ep\\
&=\pm 8\left(-\im\sigma\phi
+\left(-s+\frac{\diffordm}{2}\right)\phi\pm\frac{1}{4}(\tilde\rho\sH_p\rho_0)\phi'
+\delta\frac{\ep}{\tilde\rho+\ep}\phi\right)
\phi\tilde\rho^{-2s}(1+\ep \tilde\rho^{-1})^{-\delta},
\end{split}\end{equation*}
so
\begin{equation}\begin{split}\label{eq:rad-comm-expand}
\pm\sigma_{2s}&(\imath(P^*_\sigma C_\ep^*C_\ep-C_\ep^*C_\ep P_\sigma))\\
&\leq
-8\left(s-\frac{\diffordm}{2}+\im\sigma-\delta\right) \tilde\rho^{-2s}(1+\ep \tilde\rho^{-1})^{-\delta}\phi^2\\
&\qquad\qquad+2(\pm\tilde\rho\sH_p\rho_0)
\tilde\rho^{-2s}(1+\ep \tilde\rho^{-1})^{-\delta}\phi'\phi
.
\end{split}\end{equation}
Here the first term on the right hand side
is negative if $s-\diffordmpar/2+\im\sigma-\delta>0$
and this is the same sign as that of $\phi'$ term;
the presence of $\delta$ (needed for the regularization) is the reason for the
appearance of $m$ in the estimate.
Thus,
\begin{equation*}
\pm\imath(P^*_\sigma C_\ep^*C_\ep-C_\ep^*C_\ep P_\sigma)=-S_\ep^*(B^*B+ B_1^*B_1+B_{2,\ep}^* B_{2,\ep})S_\ep+F_\ep,
\end{equation*}
with $B,B_1,B_{2,\ep}\in\Psi^s(X)$, $B_{2,\ep}$ uniformly bounded in $\Psi^s(X)$ as
$\ep\to 0$,
$F_\ep$ uniformly bounded in
$\Psi^{2s-1}(X)$, and $\sigma_s(B)$ an elliptic multiple of $\phi(\rho_0)\tilde\rho^{-s}$.
Computing the pairing, using an extra regularization (insert a
regularizer $\Lambda_r\in\Psi^{-1}(X)$, uniformly bounded in
$\Psi^0(X)$, converging to $\Id$ in $\Psi^\delta(X)$ to justify
integration by parts, and use that $[\Lambda_r,P^*_\sigma]$ is
uniformly bounded in $\Psi^{1}(X)$, converging to $0$ strongly,
cf.\ \cite[Lemma~17.1]{Vasy:Propagation-2} and its use in \cite[Lemma~17.2]{Vasy:Propagation-2}) yields
\begin{equation}\label{eq:reg-of-pairing}
\langle \imath(P^*_\sigma C_\ep^*C_\ep-C_\ep^*C_\ep P_\sigma)u,u\rangle=
\langle \imath C_\ep^*C_\ep u,P_\sigma u\rangle-\langle\imath P_\sigma,C_\ep^*C_\ep u\rangle.
\end{equation}
Using Cauchy-Schwartz on the right hand side, a standard functional analytic argument
(see, for instance, Melrose \cite[Proof of Proposition~7 and Section~9]{RBMSpec})
gives an estimate for $Bu$, showing $u$ is in $H^s$ on the elliptic set of $B$,
provided $u$ is microlocally in $H^{s-\delta}$. A standard
inductive argument, starting with $s-\delta=m$ and improving regularity by $\leq 1/2$
in each step proves Proposition~\ref{prop:micro-out}.

For Proposition~\ref{prop:micro-in}, when applied to $P_\sigma$ in
place of $P_\sigma^*$ (so the assumption is $s<(\diffordm-\im\sigma)/2$), the argument
is similar, but we want to change the sign of the first term on the right hand side of
\eqref{eq:rad-comm-expand}, i.e.\ we want it to be positive. This
is satisfied if
$s-\diffordmpar/2+\im\sigma-\delta<0$, hence (as $\delta>0$)
if $s-\diffordmpar/2+\im\sigma<0$, so regularization is not an issue.
On the other hand, $\phi'$ now has
the wrong sign, so one needs to make an assumption on $\supp d\phi$;
one can arrange that this is in $O\setminus\Lambda$ by making $\phi$
have sufficiently small support, but identically $1$ near $0$.
Since the details are standard,
see \cite[Section~9]{RBMSpec}, we leave these to the reader. When
interchanging $P_\sigma$ and $P_\sigma^*$, we need to take into account the switch of
the sign of the principal symbol of
$\frac{1}{2\imath}(P_\sigma-P_\sigma^*)$, which causes the sign change
in front of $\im\sigma$ in the statement of the proposition.
\end{proof}

\subsection{Global estimates}
For our Fredholm results, we actually need estimates. However, these
can be easily obtained from regularity results as in e.g.\ \cite[Proof
of Theorem~26.1.7]{Hor}
by the closed graph theorem. It should be noted that of course one
really proved versions of the relevant estimates when proving
regularity, but the closed graph theorem provides a particularly
simple way of combining these (though it comes at the cost of using a
theorem which in principle is unnecessary).

So suppose $s\geq
m>(\diffordm-\im\sigma)/2$,
$u\in H^m$ and $(P_\sigma-\imath Q_\sigma)u\in H^{s-1}$.
The above results give that, first,
$\WF^s(u)$ (indeed, $\WF^{s+1}(u)$) is disjoint from the elliptic set
of $P_\sigma-\imath Q_\sigma$. Next $\Lambda_\pm$ is disjoint from
$\WF^s(u)$, hence so is a neighborhood of $\Lambda_\pm$ as the
complement of the wave
front set is open. Thus by propagation of singularities and
\eqref{eq:classical-non-trap},
taking
into account the sign of $q$ along $\Sigma_\pm$,
$\WF^s(u)\cap\Sigma_\pm=\emptyset$.
Now, by the regularity result, the inclusion map
$$
\cZ_s=\{u\in H^m:\ (P_\sigma-\imath Q_\sigma)u\in H^{s-1}\}\to H^m,
$$
in fact maps to $H^s$.

Note that $\cZ_s$ is complete with the norm
$\|u\|_{\cZ_s}^2=\|u\|_{H^m}^2+\|(P_\sigma-\imath
Q_\sigma)u\|^2_{H^{s-1}}$. Indeed, $\{u_j\}_{j=1}^\infty$ Cauchy in
$\cZ_s$ means $u_j\to u$ in $H^m$ and $ (P_\sigma-\imath
Q_\sigma)u_j\to v\in H^{s-1}$. By the first convergence, $(P_\sigma-\imath Q_\sigma)u_j\to
(P_\sigma-\imath Q_\sigma)u$ in $H^{m-2}$, thus, as $s-1\geq m-2$,
$ (P_\sigma-\imath
Q_\sigma)u_j\to v$ in $H^{m-2}$ shows $(P_\sigma-\imath
Q_\sigma)u=v\in H^{s-1}$, and thus, $(P_\sigma-\imath Q_\sigma)u_j\to
(P_\sigma-\imath Q_\sigma)u$ in $H^{s-1}$, so $u_j\to u$ in $\cZ_s$.

The graph of the inclusion map, considered as a subset of $\cZ_s\times
H^s$ is closed, for $(u_j,u_j)\to(u,v)\in\cZ_s\times H^s$ implies in
particular $u_j\to u$ and $u_j\to v$ in $H^m$, so $u=v\in\cZ_s\cap
H^s$. Correspondingly, by the closed graph theorem, the inclusion map
is continuous, i.e.
\begin{equation}\label{eq:unique-est}
\|u\|_{H^s}\leq C(\|(P_\sigma-\imath
Q_\sigma)u\|_{H^{s-\diffordm}}+\|u\|_{H^m}),\qquad u\in\cZ_s.
\end{equation}
This estimate implies that $\Ker(P_\sigma-\imath Q_\sigma)$ in
$H^s$
is finite dimensional since elements of this kernel lie in $\cZ_s$,
and since on the unit ball of this closed subspace
of $H^s$ (for $P_\sigma-\imath Q_\sigma:H^s\to H^{s-2}$ is
continuous),
$\|u\|_{H^s}\leq C\|u\|_{H^m}$, and the inclusion $H^s\to
H^m$ is compact. Further, elements of $\Ker(P_\sigma-\imath Q_\sigma)$
are in $\CI(X)$ by our regularity result, and thus this space is
independent of the choice of $s$.

On the other hand, for the adjoint operator, $P_\sigma^*+\imath Q_\sigma^*$
we have that if $s'<(\diffordm+\im\sigma)/2$
(recall
that replacing $P_\sigma$ by its
adjoint switches the sign of the principal symbol of $\frac{1}{2\imath}(P_\sigma-P_\sigma^*)$),
$u\in H^{-N}$ and $(P_\sigma^*+\imath Q_\sigma^*)u\in H^{s'-1}$ then first
$\WF^{s'}(u)$ (indeed, $\WF^{s'+1}(u)$) is disjoint from the elliptic set
of $P_\sigma^*+\imath Q_\sigma^*$.
Next, by propagation of singularities and \eqref{eq:classical-non-trap}, taking
into account the sign of $q$ along $\Sigma_\pm$, namely the sign of
the imaginary part of the principal symbol switched by taking the adjoints,
$\WF^{s'}(u)\cap(\Sigma_\pm\setminus\Lambda_\pm)=\emptyset$.
Finally, by the result at the radial points
$\Lambda_\pm$ is disjoint from
$\WF^{s'}(u)$.
Thus, the inclusion map
$$
\cW_{s'}=\{u\in H^{-N}:\ (P_\sigma^*+\imath Q_\sigma^*)u\in H^{s'-1}\}\to H^{-N},
$$
in fact maps to $H^{s'}$.
We deduce, as above, by the closed graph theorem, that
\begin{equation}\label{eq:exist-est}
\|u\|_{H^{s'}}\leq C(\|(P_\sigma^*+\imath
Q_\sigma^*)u\|_{H^{s'-\diffordm}}+\|u\|_{H^{-N}}),
\qquad u\in\cW_{s'}.
\end{equation}
As above, this estimate implies that $\Ker(P_\sigma^*+\imath
Q_\sigma^*)$ in $H^{s'}$ is
finite dimensional. Indeed, by our regularity results (elliptic
regularity, propagation of singularities, and then regularity at the
radial set) elements of $\Ker(P_\sigma^*+\imath Q_\sigma^*)$ have wave front set
in $\Lambda_+\cup\Lambda_-$ and lie in $\cap_{s'<(\diffordm+\im\sigma)/2}H^{s'}$.

The dual of $H^s$ for $s>(\diffordm-\im\sigma)/2$, is
$H^{-s}=H^{s'-\diffordm}$, $s'=\diffordm-s$, so
$s'<(\diffordm+\im\sigma)/2$ in this case,
while the dual of $H^{s-\diffordm}$, $s>(\diffordm-\im\sigma)/2$, is $H^{\diffordm-s}=H^{s'}$, with $s'=\diffordm-s<(\diffordm+\im\sigma)/2$ again.
Thus, the spaces (apart from the residual spaces $H^m$ and
$H^{-N}$,
into which the inclusion is
compact) in the left, resp.\ right, side of \eqref{eq:exist-est}, are exactly
the duals of those on the right, resp.\ left, side of \eqref{eq:unique-est}.
Thus, by a standard functional analytic argument,
see e.g.\ \cite[Proof of Theorem~26.1.7]{Hor},
namely dualization
and using the compactness
of the inclusion $H^{s'}\to H^{-N}$ for $s'>-N$, \eqref{eq:exist-est}
gives the $H^s$-solvability, $s=1-s'$ (i.e.\ we demand $u\in H^s$), of
$$
(P_\sigma-\imath Q_\sigma)u=f,\ s>(\diffordm-\im\sigma)/2,
$$
for $f$ in the annihilator (in $H^{s-1}=(H^{s'})^*$ with duality induced
by the $L^2$ inner product) of the finite dimensional subspace
$\Ker(P_\sigma^*+\imath Q_\sigma^*)$ of $H^{s'}$.

Recall from \cite[Proof of Theorem~26.1.7]{Hor}
that this argument has two parts: first for any complementary subspace
$V$ of
$\Ker(P_\sigma^*+\imath Q_\sigma^*)$ in $H^{s'}$ (i.e.\ $V$ is closed,
$V\cap\Ker(P_\sigma^*+\imath Q_\sigma^*)=\{0\}$, and
$V+\Ker(P_\sigma^*+\imath Q_\sigma^*)= H^{s'}$, e.g.\ $V$ is the
$H^{s'}$ orthocomplement of $\Ker(P_\sigma^*+\imath Q_\sigma^*)$),
one can drop $\|u\|_{H^{-N}}$ from
the right hand side of \eqref{eq:exist-est} when $u\in V\cap\cW_{s'}$
at the cost of replacing
$C$ by a larger constant $C'$. Indeed, if no $C'$ existed, one would have
a sequence $u_j\in V\cap\cW_{s'}$ such that $\|u_j\|_{H^{s'}}=1$ and
$\|(P_\sigma^*+\imath Q_\sigma^*)u_j\|_{H^{s'-\diffordm}}\to 0$, so
$(P_\sigma^*+\imath Q_\sigma^*)u_j\to 0$ in $H^{s'-\diffordm}$.
By weak
compactness of the $H^{s'}$ unit ball, there is a weakly convergent
subsequence
$u_{j_\ell}$ converging to some $u\in H^{s'}$, by the closedness (which implies
weak closedness) of $V$,
$u\in V$, so $(P_\sigma^*+\imath
Q_\sigma^*)u_{j_\ell}\to (P_\sigma^*+\imath Q_\sigma^*)u$ weakly in $H^{s'-\difford}$,
and thus $(P_\sigma^*+\imath Q_\sigma^*)u=0$ so
$u\in V\cap\Ker(P_\sigma^*+\imath Q_\sigma^*)=\{0\}$.
On the other hand, by
compactness of the inclusion $H^{s'}\to H^{-N}$, $u_{j_\ell}\to u$
strongly in $H^{-N}$, so $\{u_{j_\ell}\}$ is Cauchy in $H^{-N}$, hence from \eqref{eq:exist-est},
it is Cauchy in $H^{s'}$, so it converges to $u$
strongly in $H^{s'}$ and hence $\|u\|_{H^{s'}}=1$. This
contradicts $u=0$, completing the proof of
\begin{equation}\label{eq:exist-est-mod}
\|u\|_{H^{s'}}\leq C'\|(P_\sigma^*+\imath
Q_\sigma^*)u\|_{H^{s'-\diffordm}},\ u\in V\cap\cW_{s'}.
\end{equation}

Thus, with $s'=1-s$,
and for $f$ in the annihilator (in $H^{s-1}$, via the $L^2$-pairing)
of $\Ker(P_\sigma^*+\imath Q_\sigma^*)\subset H^{s'}$, and for $v\in V\cap\cW_{s'}$,
$$
|\langle f,v\rangle|\leq \|f\|_{H^{s-1}}\|v\|_{H^{s'}}\leq C'
\|f\|_{H^{s-1}} \|(P_\sigma^*+\imath Q_\sigma^*)v\|_{H^{s'-1}}.
$$
As adding
an element of $\Ker(P_\sigma^*+\imath Q_\sigma^*)$ to $v$ does not
change either side, the inequality holds for all $v\in\cW_{s'}\subset H^{s'}$. Thus,
the conjugate-linear map $(P_\sigma^*+\imath Q_\sigma^*)v\mapsto \langle
f,v\rangle$, $v\in \cW_{s'}$, which is well-defined, is continuous from
$\Ran_{\cW_{s'}}(P_\sigma^*+\imath Q_\sigma^*)\subset H^{s'-1}$ to $\Cx$,
and by the Hahn-Banach theorem can be extended to a continuous
conjugate linear functional $\ell$ on $H^{s'-1}=(H^s)^*$, so there exists
$u\in H^s$ such that $\langle u,\phi\rangle=\ell(\phi)$ for $\phi\in
H^{s'-1}$.
In particular, when $\phi=(P_\sigma^*+\imath Q_\sigma^*)\psi$,
$\psi\in\CI(X)\subset\cW_{s'}$,
$$
\langle u,(P_\sigma^*+\imath Q_\sigma^*)\psi\rangle=\ell(\phi)=\langle
f,\psi\rangle,
$$
so $(P_\sigma-\imath Q_\sigma)u=f$ as claimed.

In order to set up Fredholm theory, let $\tilde P$ be any
operator with principal symbol $p-\imath q$; e.g.\ $\tilde P$ is
$P_{\sigma_0}-\imath Q_{\sigma_0}$ for some
$\sigma_0$. Then consider
\begin{equation}\label{eq:XY-def}
\cX^s=\{u\in H^s:\ \tilde Pu\in H^{s-\diffordm}\},\ \cY^s=H^{s-\diffordm},
\end{equation}
with
$$
\|u\|_{\cX^s}^2=\|u\|_{H^s}^2+\|\tilde Pu\|^2_{H^{s-\diffordm}}.
$$
Note that the $\cZ_s$-norm is equivalent to the $\cX^s$-norm, and
$\cZ_s=\cX^s$, by \eqref{eq:unique-est} and the preceding discussion.
Note that $\cX^s$ only depends on the principal symbol of $\tilde P$. Moreover,
$\CI(X)$ is dense in $\cX^s$; this follows by considering $R_\ep\in\Psi^{-\infty}(X)$,
$\ep>0$,
such that $R_\ep\to\Id$ in $\Psi^\delta(X)$ for $\delta>0$, $R_\ep$ uniformly
bounded in $\Psi^0(X)$; thus $R_\ep\to\Id$ strongly (but not in the operator norm
topology) on $H^s$ and $H^{s-\diffordm}$.
Then for $u\in \cX^s$, $R_\ep u\in\CI(X)$ for $\ep>0$, $R_\ep u\to u$
in $H^s$ and $\tilde P R_\ep u=R_\ep\tilde Pu+[\tilde P,R_\ep]u$, so the first
term on the right converges to $\tilde Pu$ in $H^{s-\diffordm}$, while $[\tilde P,R_{\ep}]$
is uniformly bounded in $\Psi^{\diffordm}(X)$, converging to $0$ in $\Psi^{\diffordm+\delta}(X)$ for
$\delta>0$, so converging to $0$ strongly as a map $H^s\to H^{s-\diffordm}$. Thus,
$[\tilde P,R_\ep]u\to 0$ in $H^{s-\diffordm}$, and
we conclude that $R_\ep u\to u$ in $\cX^s$.
(In fact, $\cX^s$ is a first-order coisotropic space, more general
function spaces of this nature are discussed
by Melrose, Vasy and Wunsch in \cite[Appendix~A]{Melrose-Vasy-Wunsch:Corners}.)

With these preliminaries,
$$
P_\sigma-\imath Q_\sigma:\cX^s\to \cY^s
$$
is bounded for each $\sigma$ with $s\geq m>(\diffordm-\im\sigma)/2$,
and is an analytic family of bounded operators in this half-plane of
$\sigma$'s. Further, it is Fredholm for each $\sigma$: the kernel in
$\cX^s$ is finite dimensional, and it surjects onto the annihilator in $H^{s-1}$ of
the (finite dimensional) kernel of
$P_\sigma^*+\imath Q_\sigma^*$ in $H^{1-s}$, which thus has finite
codimension, and is closed, since for $f$ in this space there exists
$u\in H^s$ with $(P_\sigma-\imath Q_\sigma)u=f$, and thus $u\in \cX^s$.
Restating this as a theorem:

\begin{thm}\label{thm:classical-absorb}
Let $P_\sigma$, $Q_\sigma$ be as above,
and $\cX^s$, $\cY^s$ as in \eqref{eq:XY-def}.
Then
$$
P_\sigma-\imath Q_\sigma:\cX^s\to\cY^s
$$
is an analytic family of Fredholm operators on
\begin{equation}\label{eq:Cx-s-def}
\Cx_s=\{\sigma\in\Cx:\ \im\sigma>\diffordm-2s\}.
\end{equation}
\end{thm}

Thus, analytic Fredholm theory applies, giving meromorphy of the inverse
provided the inverse exists for a particular value of $\sigma$.

\begin{rem}\label{rem:dual-Fredholm}
Note that the Fredholm property means that $P_\sigma^*+\imath Q_\sigma^*$ is
also Fredholm on the dual spaces; this can also be seen directly from the estimates.
The analogue of this remark also applies to the semiclassical discussion below.
\end{rem}

\subsection{Semiclassical estimates}
There are semiclassical estimates completely analogous to those in the
classical setting; we again phrase these as wave front set statements.
Let $\Hh^s$ denote the semiclassical Sobolev space of order $s$, i.e.\
as a function space this is the space of functions $(u_h)_{h\in I}$,
$I\subset (0,1]_h$ with values
in the standard Sobolev space $H^s$, with $A_h u_h$ bounded in $L^2$
for an elliptic, semiclassically elliptic, operator
$A_h\in\Psih^s(X)$.
(Note that $u_h$ need not be defined for all $h\in(0,1]$; we suppress
$I$ from the notation.)
Let
$$\WFh^{s,r}(u)\subset\pa
(\overline{T}^*X\times[0,1)_h)=S^*X\times[0,1)_h\cup T^*X\times\{0\}_h
$$
denote the semiclassical wave front set of a polynomially
bounded family of distributions, i.e.\ $u=(u_h)_{h\in I}$,
$I\subset(0,1]$, satisfying $u_h$ is uniformly
bounded in $h^{-N}\Hh^{-N}$ for some $N$.
This is defined as follows: we say that
$\alpha\notin\WFh^{s,r}(u)$ if there exists
$A\in\Psih^0(X)$ elliptic at $\alpha$ such that $Au\in
h^r\Hh^s$. Note that, in view of the description of the symbols in
Section~\ref{sec:notation},
ellipticity at $\alpha$ means the ellipticity of
$\sigma_\semi(A_h)$ if $\alpha\in T^*X\times\{0\}$, that of
$\sigma_0(A_h)$ if $\alpha\in S^*X\times(0,1)$, and that of either
(and thus both, in view of the compatibility of these symbols)
of these when $\alpha\in S^*X\times\{0\}$.
The
semiclassical wave front set captures global estimates: if $u$ is
polynomially bounded and $\WFh^{s,r}(u)=\emptyset$, then $u\in
h^r\Hh^s$. 

Elliptic
regularity states that
$$
P_{h,z}-\imath Q_{h,z}\ \text{elliptic at}\ \alpha,\
\alpha\notin\WFh^{s-2,0}((P_{h,z}-\imath Q_{h,z})u)\Rightarrow
\alpha\notin\WFh^{s,0}(u).
$$
Thus, $(P_{h,z}-\imath Q_{h,z})\in \Hh^{s-2}$ and
$p_{\semi,z}-\imath q_{\semi,z}$ is elliptic at $\alpha$ then $\alpha\notin\WFh^s(u)$.

We also have real principal type propagation:
$$
\WFh^{s,-1}(u)\setminus(\WFh^{s-1,0}((P_{h,z}-\imath Q_{h,z})u)\cup\supp q_{\semi,z})
$$
is a union of maximally extended bicharacteristics of $\sH_p$ in the
characteristic set $\Sigma_{\semi,z}=\{p_{\semi,z}=0\}$ of $P_{h,z}$. Put it differently,
$$
\alpha\notin\WF^{s,-1}(u)\cup \WF^{s-1,0}((P_{h,z}-\imath Q_{h,z})u)\cup\supp
q_{\semi,z}\Rightarrow \tilde\gamma(\alpha)\cap\WF^{s,-1}(u)=\emptyset,
$$
where $\tilde\gamma(\alpha)$ is the component of the bicharacteristic $\gamma(\alpha)$
of $p_{\semi,z}$ in the complement of $\WF^{s-1,0}((P_{h,z}-\imath Q_{h,z})u)\cup\supp
q_{\semi,z}$. If $(P_{h,z}-\imath Q_{h,z})u\in H^{s-1}$, then $\WF^{s-1,0}((P_{h,z}-\imath
Q_{h,z})u)=\emptyset$ can be dropped from all statements above; if
$q_{\semi,z}=0$ one can thus replace $\tilde\gamma$ by $\gamma$.

In general, the result does not hold for non-zero $q_{\semi,z}$. However, it
holds in one direction (backward/forward) of propagation along $\sH_{p_{\semi,z}}$
if $q_{\semi_z}$ has the correct sign. Thus, with
$\tilde\gamma_\pm(\alpha)$ a forward ($+$) or backward ($-$) bicharacteristic from
$\alpha$ defined on an interval,
if $\pm q_{\semi,z}\geq 0$ on a neighborhood of
$\tilde\gamma_\pm(\alpha)$ then
\begin{equation*}\begin{split}
\alpha\notin\WFh^{s,-1}(u)&\Mand\tilde\gamma_\pm(\alpha)\cap\WFh^{s-1,0}((P_{\semi,z}-\imath
Q_{\semi,z})u)=\emptyset\\
&\Rightarrow \tilde\gamma_\pm(\alpha)\cap\WFh^{s,-1}(u)=\emptyset,
\end{split}\end{equation*}
i.e.\ one can propagate regularity forward if $q_{\semi,z}\geq 0$, backward if
$q_{\semi,z}\leq 0$; 
see \cite{Nonnenmacher-Zworski:Quantum}
and \cite{Datchev-Vasy:Gluing-prop}. Again, for $P_{\semi,z}^*+\imath
Q_{\semi,z}^*$ the directions are reversed,
i.e.\ one can propagate regularity forward if $q_{\semi,z}\leq 0$, backward if
$q_{\semi,z}\geq 0$.

A semiclassical version of Melrose's regularity result was proved by
Zworski and the author in the asymptotically Euclidean setting, \cite{Vasy-Zworski:Semiclassical}.
We need a more general result, which is an easy adaptation:

\begin{prop}\label{prop:micro-out-semi}
Suppose $s\geq m>(\diffordm-\im\sigma)/2$, and
$\WFh^{m,-N}(u)\cap L_\pm=\emptyset$ for some $N$. Then
$$
L_\pm\cap\WF^{s-1,0}(P_{\semi,z} u)=\emptyset\Rightarrow L_\pm\cap\WF^{s,-1}(u)=\emptyset.
$$
\end{prop}

Again, at the cost of changing regularity, one can
propagate estimate {\em towards} $L_\pm$.

\begin{prop}\label{prop:micro-in-semi}
For $s<(\diffordm+\im\sigma)/2$, and $O$ a neighborhood of
$L_\pm$,
$$
\WFh^{s,-1}(u)\cap (O\setminus L_\pm)=\emptyset,\ \WFh^{s-1,0}(P_{\semi,z}^*
u)\cap L_\pm=\emptyset\Rightarrow \WFh^{s,-1}(u)\cap L_\pm=\emptyset.
$$
\end{prop}

\begin{proof}
We just need to localize in $\tilde\rho$ in addition to $\rho_0$;
such a localization in the
classical setting is implied by working on $S^*X$ or with homogeneous
symbols.
We achieve this by modifying the localizer $\phi$ in the commutant
constructed in the proof of Propositions~\ref{prop:micro-out}-\ref{prop:micro-in}.
As already remarked, the proof is
much like at radial points in semiclassical scattering on asymptotically
Euclidean spaces, studied by Vasy and Zworski \cite{Vasy-Zworski:Semiclassical}, but
we need to be more careful about localization in $\rho_0$ and $\tilde\rho$ as we
are assuming less about the structure.

First, note that $L_\pm$ is defined by $\tilde\rho=0$, $\rho_0=0$, so $\tilde\rho^2+
\rho_0$ is a quadratic defining function of $L_\pm$.
Thus, let $\phi\in\CI_c(\RR)$ be identically $1$ near $0$,
$\phi'\leq 0$ and
$\phi$ supported sufficiently close to $0$ so that
\begin{equation*}\begin{split}
&\tilde\rho^2+\rho_0\in\supp d\phi\Rightarrow
\pm \tilde\rho(\sH_p\rho_0+2\tilde\rho\sH_p\tilde\rho)>0
\end{split}\end{equation*}
and
\begin{equation*}\begin{split}
&\tilde\rho^2+\rho_0\in\supp \phi\Rightarrow
\pm \tilde\rho\sH_p\tilde\rho>0.
\end{split}\end{equation*}
Such $\phi$ exists by
\eqref{eq:weight-definite} and \eqref{eq:rho-0-property} as
$$
\pm\tilde\rho(\sH_p\rho_0+2\tilde\rho\sH_p\tilde\rho)\geq 8\rho_0+
8\tilde\rho^2-\cO((\tilde\rho^2+
\rho_0)^{3/2}).
$$
Then let $c$ be given by
$$
c=\phi(\rho_0+\tilde\rho^2)\tilde\rho^{-s+\diffordmpar/2},\qquad
c_\ep=c(1+\ep \tilde\rho^{-1})^{-\delta}.
$$
The rest of the proof proceeds exactly as for
Propositions~\ref{prop:micro-out}-\ref{prop:micro-in}.
\end{proof}

Suppose now that $p_{\semi,z}$ is semiclassically non-trapping, as
discussed at the end of Section~\ref{sec:spaces}.
Suppose
again that $s\geq
m>(\diffordm-\im\sigma)/2$,
$h^N u_h$ is bounded in $\Hh^m$ and $(P_{h,z}-\imath Q_{h,z})u_h\in \Hh^{s-1}$.
The above results give that, first,
$\WFh^{s,-1}(u)$ (indeed, $\WFh^{s+1,0}(u)$) is disjoint from the elliptic set
of $P_{h,z}-\imath Q_{h,z}$. Next we see that $L_\pm$ is disjoint from
$\WFh^{s,-1}(u)$, hence so is a neighborhood of $L_\pm$.
Thus by propagation of singularities and the semiclassically
non-trapping property, taking
into account the sign of $q$ along $\Sigma_{\semi,\pm}$,
$\WFh^{s,-1}(u)\cap\Sigma_{\semi,\pm}=\emptyset$.
In summary, $\WFh^{s,-1}(u)=\emptyset$, i.e.\ $hu_h$ is bounded in
$\Hh^s$, i.e.
\begin{equation}\label{eq:glob-semi-reg}
h^N u_h\ \text{bounded in}\ \Hh^m,\ (P_{h,z}-\imath Q_{h,z})u_h\in
\Hh^{s-1}\Rightarrow
hu_h\in \Hh^s.
\end{equation}

Now suppose that for a decreasing sequence $h_j\to 0$,
$w_h\in\Ker(P_{h,z}-\imath Q_{h,z})$ and $\|w_h\|_{\Hh^s}=1$. Then for
any $N$, $u_h=h^{-N}w_h$ satisfies the above hypotheses, and we deduce
that $hu_h$ is uniformly bounded in $\Hh^s$, i.e.\ $h^{-N+1}w_h$ is
uniformly bounded in $\Hh^s$. But for $N>1$ this contradicts that
$\|w_h\|_{\Hh^s}=1$, so such a sequence $h_j$ does not
exist. Therefore $\Ker(P_{h,z}-\imath Q_{h,z})=\{0\}$ for sufficiently
small $h$.

Using semiclassical propagation of singularities in the reverse direction, much as
we did in the previous section, we deduce that $\Ker(P_{h,z}^*+\imath
Q_{h,z}^*)=\{0\}$ for $h$ sufficiently small. Since $P_{h,z}-\imath
Q_{h,z}:\cX^s\to\cY^s$ is Fredholm, we deduce immediately that there
exists $h_0$ such that it is
invertible for $h<h_0$.

In order to obtain uniform estimates for $(P_{h,z}-\imath
Q_{h,z})^{-1}$ as $h\to 0$, it is convenient to `renormalize' the
problem to make the function spaces (and their norms) independent of
$h$
so that one can use
the uniform boundedness principle. (Again, this could have been
avoided if we had just stated the estimates uniformly in $u$ as well,
much like the closed graph theorem could have been avoided in the
previous section.) So for $r\in\RR$ let $\Lambda^r_h\in\Psi^r_h$ be
elliptic and invertible, and let
$$
P^s_{h,z}-\imath Q^s_{h,z}=\Lambda^{s-1}_h (P_{h,z}-\imath Q_{h,z})\Lambda^s_h.
$$
Then, with $\tilde P=P^s_{h_0,z_0}\in\Psi^1(X)$, for instance,
independent of $h$,
$$
\cX=\{u\in L^2:\ \tilde Pu\in L^2\},\ \cY=L^2,
$$
$P^s_{h,z}-\imath Q^s_{h,z}:\cX\to\cY$ is invertible for $h<h_0$ by
the above observations. Let $j:\cX\to \cZ=L^2$ be the inclusion map. Then
$$
j\circ h(P^s_{h,z}-\imath Q^s_{h,z})^{-1}:\cY\to \cZ
$$
is continuous for each $h<h_0$.

We claim that for each (non-zero) $f\in\cY$, $\{\|h(P^s_{h,z}-\imath
Q^s_{h,z})^{-1}f\|_{L^2}:\ h<h_0\}$ is bounded. Indeed, let $v_h=h(P^s_{h,z}-\imath
Q^s_{h,z})^{-1}f$, so we need to show that $v_h$ is bounded. Suppose first that $h v_h$ is not bounded,
so consider a sequence $h_j$ with $h_j \|v_{h_j}\|_{L^2}\geq
1$.
Then let $u_h=h^{-2}v_h/\|v_h\|_{L^2}$,  $h\in\{h_j:\ j\in\Nat\}$,
so $h^2u_h$ is bounded in
$L^2$, so $u_h$ is in particular polynomially bounded in $L^2$.
Also, $(P^s_{h,z}-\imath Q^s_{h,z})u_h=h^{-1}f/\|v_h\|_{L^2}$ is
bounded in $L^2$ as $\|v_h\|\geq h^{-1}$.
Thus, by \eqref{eq:glob-semi-reg},
$hu_h$ is bounded in $L^2$, i.e.\ $h^{-1}v_h/\|v_h\|_{L^2}$ is
bounded, which is a contradiction, showing that $hv_h$ is bounded.
Thus, introducing a new $u_h$, namely $u_h=h^{-1}v_h$, $u_h$ is polynomially bounded, and
$(P^s_{h,z}-\imath Q^s_{h,z})u_h=f$ is bounded, so, by
\eqref{eq:glob-semi-reg},
$hu_h=v_h$ is
bounded as claimed.

Thus, by the uniform boundedness principle, $j\circ h(P^s_{h,z}-\imath
Q^s_{h,z})^{-1}$ is equicontinuous. Undoing the transformation, we
deduce that
$$
\|(P_{h,z}-\imath Q_{h,z})^{-1}f\|_{\Hh^s}\leq
Ch^{-1}\|f\|_{\Hh^{s-1}},
$$
which is exactly the high energy estimate we were after.

Our arguments were under the assumption of semiclassical
non-trapping. As discussed in Subsections~\ref{subsec:global-semi} and
\ref{subsec:complex-absorb-dS},
this always holds in sectors $\delta|\re\sigma|<\im\sigma<\delta_0|\re
\sigma|$ (with $Q_\sigma$ supported in $\mu<0$!)
since $P_{h,z}-\imath Q_{h,z}$ is actually semiclassically
elliptic then. In particular this gives the meromorphy of
$P_\sigma-\imath Q_\sigma$ by giving invertibility of large $\sigma$
in such a sector. Rephrasing in the large parameter notation, using
$\sigma$ instead of $h$,

\begin{thm}\label{thm:classical-absorb-sector}
Let $P_\sigma$, $Q_\sigma$, $\Cx_s$ be as above,
and $\cX^s$, $\cY^s$ as in \eqref{eq:XY-def}.
Then, for $\sigma\in\Cx_s$,
$$
P_\sigma-\imath Q_\sigma:\cX^s\to\cY^s
$$
has a meromorphic inverse
$$
R(\sigma):\cY^s\to\cX^s.
$$
Moreover, there is $\delta_0>0$ such that for all $\delta\in(0,\delta_0)$ there is $\sigma_0>0$
such that $R(\sigma)$ is invertible in
$$
\{\sigma:\ \delta|\re\sigma|<\im\sigma<\delta_0|\re \sigma|,\
|\re\sigma|>\sigma_0\},
$$
and
{\em non-trapping estimates} hold:
$$
\|R(\sigma)f\|_{H^s_{|\sigma|^{-1}}}
\leq C'|\sigma|^{-\diffordm}\|f\|_{H^{s-1}_{|\sigma|^{-1}}}.
$$
\end{thm}

If the metric $g_0$ is non-trapping then $p_{\semi,z}-\imath
q_{\semi,z}$ and its complex conjugate are semiclassically
non-trapping by Subsection~\ref{subsec:global-semi}, so
the high energy estimates
are then applicable in half-planes $\im\sigma<-C$, i.e.\ half-planes
$\im z\geq -Ch$. The same holds for trapping $g_0$ provided that we
add a complex absorbing operator near the trapping, as discussed in
Subsection~\ref{subsec:complex-absorb-dS}.

Translated into the classical setting this gives

\begin{thm}\label{thm:classical-absorb-strip}
Let $P_\sigma$, $Q_\sigma$, $\Cx_s$, $\delta_0>0$ be as above, in particular
semiclassically non-trapping,
and $\cX^s$, $\cY^s$ as in \eqref{eq:XY-def}. Let $C>0$. Then there exists $\sigma_0$
such that
$$
R(\sigma):\cY^s\to\cX^s,
$$
is holomorphic
in $\{\sigma:\ -C<\im\sigma<\delta_0|\re\sigma|,\ |\re\sigma|>\sigma_0\}$, assumed to be a subset
of $\Cx_s$,
and {\em non-trapping estimates}
$$
\|R(\sigma)f\|_{H^s_{|\sigma|^{-1}}}
\leq C'|\sigma|^{-\diffordm}\|f\|_{H^{s-\diffordm}_{|\sigma|^{-1}}}
$$
hold.
For $s=1$ this states that for $|\re\sigma|>\sigma_0$, $\im\sigma>-C$,
$$
\|R(\sigma)f\|_{L^2}^2+|\sigma|^{-2}\|dR(\sigma)\|_{L^2}^2
\leq C''|\sigma|^{-2}\|f\|_{L^2}^2.
$$
\end{thm}

While we stated just the global results here, one also has microlocal estimates
for the solution. In particular we have the following, stated in the semiclassical
language, as immediate from the estimates used to derive from the Fredholm
property:

\begin{thm}\label{thm:semicl-outg}
Let $P_\sigma$, $Q_\sigma$, $\Cx_s$ be as above, in particular
semiclassically non-trapping,
and $\cX^s$, $\cY^s$ as in \eqref{eq:XY-def}. 

For $\re z>0$ and $s'>s$, the resolvent $R_{h,z}$ is
{\em semiclassically outgoing with a loss of $h^{-1}$} in
the sense that if $\alpha\in \overline{T}^*X\cap\Sigma_{\semi,\pm}$, and if for the
backward ($-$), resp.\ forward ($+$), bicharacteristic $\gamma_\mp$, from $\alpha$,
$\WFh^{s'-\diffordm,0}(f)\cap\overline{\gamma_\mp}=\emptyset$ then $\alpha\notin\WFh^{s',-1}(R_{h,z}f)$.

In fact, for any $s'\in\RR$, the resolvent $R_{h,z}$ extends to $f\in H^{s'}_h(X)$,
with non-trapping bounds, provided that $\WFh^{s,0}(f)\cap(L_+\cup L_-)=\emptyset$.
The semiclassically outgoing with a loss of $h^{-1}$ result holds for such $f$ and
$s'$ as well.
\end{thm}

\begin{proof}
The only part that is not immediate by what has been discussed is the last
claim. This follows immediately, however, by microlocal solvability in arbitrary
ordered Sobolev spaces away from the radial points (i.e.\ solvability modulo
$\CI$, with semiclassical estimates), combined with our preceding
results to deal
with this smooth remainder plus the contribution near $L_+\cup L_-$, which
are assumed to be in $H^s_h(X)$.
\end{proof}

This result is needed for gluing constructions as in \cite{Datchev-Vasy:Gluing-prop},
namely polynomially bounded trapping with appropriate microlocal geometry
can be glued to our resolvent. Furthermore,
it gives non-trapping estimates microlocally away from the trapped set provided
the overall (trapped) resolvent is polynomially bounded as shown by Datchev
and Vasy \cite{Datchev-Vasy:Trapped}.

\section{Results in the conformally compact setting}\label{sec:conf-comp-results}
We now state our results in the original conformally compact setting.
Without the non-trapping estimate,
these are a special case of a result of Mazzeo and Melrose
\cite{Mazzeo-Melrose:Meromorphic}, with improvements by
Guillarmou \cite{Guillarmou:Meromorphic}, with `special' meaning that
evenness is
assumed. If the space is asymptotic to actual hyperbolic space,
the non-trapping estimate is a slightly stronger version of
the estimate of \cite{Melrose-SaBarreto-Vasy:Semiclassical}, where it
is shown by a parametrix construction; here conformal infinity can
have arbitrary geometry.
The point is thus that first, we do not
need the machinery of the zero calculus here, second, we do have
non-trapping high energy estimates in general (and without a
parametrix construction), and third,  we add the semiclassically
outgoing property which is useful for resolvent gluing, including
for proving non-trapping bounds microlocally away from trapping, provided
the latter is mild, as shown by Datchev and Vasy \cite{Datchev-Vasy:Gluing-prop,
Datchev-Vasy:Trapped}.

\begin{thm}\label{thm:conf-compact-high}
Suppose that $(X_0,g_0)$ is an $\dimnpar$-dimensional manifold with boundary with
an even conformally compact metric and boundary defining function $x$. Let
$X_{0,\even}$ denote the even version of $X_0$, i.e.\ with the boundary defining
function replaced by its square with respect to a decomposition in which
$g_0$ is even. Then
the inverse of
$$
\Delta_{g_0}-\left(\frac{\dimnm}{2}\right)^2-\sigma^2,
$$
written as $\cR(\sigma):L^2\to L^2$,
has a meromorphic continuation from
$\im\sigma\gg0$ to $\Cx$,
$$
\cR(\sigma):\dCI(X_0)\to\dist(X_0),
$$
with poles with finite rank residues. If in addition
$(X_0,g_0)$ is non-trapping, then, with $\phi$ as in
Subsection~\ref{subsec:Laplacian-extension}, and for suitable $\delta_0>0$,
non-trapping estimates hold in every region $-C<\im\sigma<\delta_0|\re\sigma|$, $|\re\sigma|\gg 0$: for
$s>\frac{1}{2}+C$,
\begin{equation}\label{eq:non-trap-conf-comp}
\|x^{-(\dimnm)/2}e^{\imath\sigma\phi} \cR(\sigma)f\|_{H^s_{|\sigma|^{-1}}(X_{0,\even})}
\leq \tilde C|\sigma|^{-1}\|x^{-(\dimnppppar)/2}e^{\imath\sigma\phi}f\|_{H^{s-1}_{|\sigma|^{-1}}(X_{0,\even})}.
\end{equation}
If $f$ is supported in $X_0^\circ$,
the $s-1$ norm on $f$ can be replaced by the $s-2$ norm.

Furthermore, for $\re z>0$, $\im z=\cO(h)$,
the resolvent $\cR(h^{-1}z)$ is {\em semiclassically outgoing with a loss
of $h^{-1}$} in the sense that
if $f$ has compact support in $X_0^\circ$, $\alpha\in T^*X$ is in
the semiclassical characteristic set and if $\WFh^{s-1,0}(f)$ is disjoint from
the backward bicharacteristic from $\alpha$, then
$\alpha\notin\WFh^{s,-1}(\cR(h^{-1}z)f)$.
\end{thm}

We remark that although in order
to go through without changes, our methods require the
evenness property, it is not hard to deduce more restricted results without
this. Essentially one would have operators with coefficients that have a conormal
singularity at the event horizon; as long as this is sufficiently mild relative
to what is required for the analysis, it does not affect the results. The problems
arise for the analytic continuation, when one needs strong function spaces
($H^s$ with $s$ large); these are not preserved when one multiplies by the
singular coefficients.

\begin{proof}
All of the results of Section~\ref{sec:microlocal} apply.

By self-adjointness and positivity of $\Delta_{g_0}$ and as
$\dCI(X_0)$ is in its domain,
$$
\left(\Delta_{g_0}-\sigma^2-\left(\frac{\dimnm}{2}\right)^2\right) u=f\in \dCI(X_0)
$$
has a unique solution $u=\cR(\sigma)f\in L^2(X_0,|dg_0|)$ when
when $\im\sigma\gg0$.
On the other hand, let $\phi$ be as in
Subsection~\ref{subsec:Laplacian-extension},
so $e^\phi=\mu^{1/2}(1+\mu)^{-1/4}$ near $\mu=0$ (so $e^\phi\sim x$ there),
$\tilde f_0=e^{\imath\sigma\phi}x^{-\dimnppar/2}x^{-1} f$
in $\mu\geq 0$, and $\tilde f_0$ still
vanishes to infinite order at $\mu=0$. Let $\tilde f$ be an arbitrary smooth
extension of $\tilde f_0$ to the compact manifold $X$ on which
$P_{\sigman}-\imath Q_{\sigman}$ is defined.
Let $\tilde u=(P_{\sigman}-\imath Q_{\sigman})^{-1}\tilde f$,
with $(P_{\sigman}-\imath Q_{\sigman})^{-1}$ given by our results in
Section~\ref{sec:microlocal}; this satisfies
$(P_{\sigman}-\imath Q_{\sigman})\tilde u=\tilde f$ and $\tilde u\in\CI(X)$.
Thus, $u'=e^{-\imath\sigma\phi} x^{\dimnppar/2}x ^{-1} \tilde u|_{\mu>0}$ satisfies
$u'\in x^{(\dimnm)/2}e^{-\imath\sigma\phi}\CI(X_0)$, and
$$
\left(\Delta_{g_0}-\sigma^2-\left(\frac{\dimnm}{2}\right)^2\right)u'=f
$$
by \eqref{eq:ah-final-conj-form} and \eqref{eq:P-semicl-form}
(as $Q_\sigma$ is supported in $\mu<0$).
Since $u'\in L^2(X_0,|dg_0|)$ for $\im\sigma>0$,
by the aforementioned
uniqueness, $u=u'$.

To make the extension from $X_{0,\even}$ to $X$ more systematic,
let $E_s:H^s(X_{0,\even})\to H^s(X)$ be a continuous extension operator,
$R_s:H^s(X)\to H^s(X_{0,\even})$ the restriction map. Then, as we have just
seen, for $f\in\dCI(X_0)$,
\begin{equation}\label{eq:hyp-res-repr}
\cR(\sigma)f=e^{-\imath\sigma\phi} x^{\dimnppar/2}x ^{-1}
R_s(P_{\sigman}-\imath Q_{\sigman})^{-1}E_{s-1}
e^{\imath\sigma\phi}x^{-\dimnppar/2}x^{-1} f.
\end{equation}
While, for the sake of simplicity, $Q_{\sigman}$ is constructed in
Subsection~\ref{subsec:complex-absorb-dS} in such a manner that it is
not holomorphic in all of $\im\sigma>-C$ due to a cut in the upper
half plane, this cut can be moved outside any fixed compact subset, so
taking into account that $\cR(\sigma)$ is independent of the choice of $Q_{\sigman}$,
the theorem
follows immediately from the
results
of Section~\ref{sec:microlocal}.
\end{proof}

Our argument proves that every pole of
$\cR(\sigma)$ is a pole of $(P_\sigma-\imath Q_\sigma)^{-1}$ (for otherwise
\eqref{eq:hyp-res-repr}
would show $\cR(\sigma)$ does not have a pole either),
but it is possible for $(P_\sigma-\imath Q_\sigma)^{-1}$ to have poles which are not poles of
$\cR(\sigma)$. However, in the latter case, the Laurent coefficients of
$(P_\sigma-\imath Q_\sigma)^{-1}$ would be annihilated by multiplication by $R_s$ from the
left, i.e.\ the resonant states (which are smooth) would be supported in $\mu\leq 0$,
in particular vanish to infinite order at $\mu=0$.

In fact, a stronger statement can be made: by a calculation completely analogous
to what we just performed, we can easily see that in $\mu<0$, $P_\sigma$
is a conjugate (times a power of $\mu$) of a
Klein-Gordon-type operator on $\dimnpar$-dimensional de Sitter space
with $\mu=0$ being the boundary (i.e.\ where time goes to infinity). Thus,
if $\sigma$ is not a pole of $\cR(\sigma)$ and
$(P_\sigma-\imath Q_\sigma)\tilde u=0$ then one would have a
solution $u$ of this Klein-Gordon-type equation near $\mu=0$, i.e.\ infinity,
that rapidly vanishes at infinity. It is shown
in \cite[Proposition~5.3]{Vasy:De-Sitter} by a Carleman-type estimate that this
cannot happen; although there $\sigma^2\in\RR$ is assumed, the argument given
there goes through almost verbatim in general. Thus, if $Q_\sigma$ is supported
in $\mu<c$, $c<0$, then $\tilde u$ is also supported in $\mu<c$. This argument
can be iterated for Laurent coefficients of higher order poles; their range (which
is finite dimensional) contains only functions supported
in $\mu<c$.

\begin{rem}\label{rem:ah-decay}
We now return to our previous remarks regarding the fact that our solution disallows
the conormal singularities $(\mu\pm i0)^{\imath\sigman}$ from the perspective
of conformally compact spaces of dimension $\dimn$. Recalling that
$\mu=x^2$, the
two indicial roots on these spaces
correspond to the asymptotics $\mu^{\pm\imath\sigman/2+(\dimnm)/4}$ in $\mu>0$.
Thus for the operator
$$
\mu^{-1/2}\mu^{\imath\sigman/2-\dimnppar/4}
(\Delta_{g_0}-\frac{(\dimnm)^2}{4}-\sigman^2)
\mu^{-\imath\sigman/2+\dimnppar/4}\mu^{-1/2},
$$
or indeed $P_\sigma$, they correspond to
$$
\left(\mu^{-\imath\sigman/2+\dimnppar/4}\mu^{-1/2}\right)^{-1}\mu^{\pm\imath\sigman/2+(\dimnm)/4}=\mu^{\imath\sigman/2\pm\imath\sigman/2}.
$$
Here the indicial root $\mu^0=1$ corresponds to the smooth solutions
we construct for $P_\sigma$, while $\mu^{\imath\sigman}$ corresponds
to the conormal behavior we rule out. Back to the original Laplacian, thus,
$\mu^{-\imath\sigman/2+(\dimnm)/4}$ is the allowed asymptotics
and $\mu^{\imath\sigman/2+(\dimnm)/4}$ is the disallowed one. Notice that $\re\imath\sigma
=-\im\sigma$, so the disallowed solution is growing at $\mu=0$
relative to the allowed one, as expected in the physical half plane, and the
behavior reverses when $\im\sigma<0$. Thus, in the original asymptotically
hyperbolic picture one has to distinguish two different rates of growths, whose
relative size changes.
On the other hand, in our approach, we rule out the singular solution and allow
the non-singular (smooth one), so there is no change in behavior at all
for the analytic continuation.
\end{rem}

\begin{rem}\label{rem:asymp-dS}
For {\em even} asymptotically de Sitter metrics on an $\dimnpar$-dimensional
manifold $X'_0$ with boundary, the methods
for asymptotically hyperbolic spaces work, except $P_\sigma-\imath Q_\sigma$
and $P_\sigma^*+\imath Q_\sigma^*$ switch roles, which does not affect
Fredholm properties, see Remark~\ref{rem:dual-Fredholm}.
Again, evenness means that
we may choose a product decomposition near the boundary such that
\begin{equation}\label{eq:dS-g-0-prod}
g_0=\frac{dx^2-h}{x^2}
\end{equation}
there,
where $h$ is an even family of Riemannian metrics; as above, we take $x$ to be a
globally defined boundary defining function. Then with $\mut=x^2$, so $\mut>0$
is the Lorentzian region, $\overline{\sigma}$ in place of $\sigma$ (recalling that
our aim is to get to $P_\sigma^*+\imath Q_\sigma^*$) the above calculations for
$\Box_{g_0}-\frac{(\dimnm)^2}{4}-\overline{\sigma}^2$ in place of
$\Delta_{g_0}-\frac{(\dimnm)^2}{4}-\sigma^2$ leading to
\eqref{eq:ah-prefinal-conj-form} all go through with $\mu$ replaced by $\mut$,
$\sigma$ replaced by $\overline{\sigma}$ and
$\Delta_h$ replaced by $-\Delta_h$. Letting $\mu=-\mut$, and conjugating by
$(1+\mu)^{\imath\overline{\sigma}/4}$ as above, yields
\begin{equation}\begin{split}\label{eq:dS-final-conj-form}
&-4\mu D_\mu^2+4\overline{\sigman} D_\mu+\overline{\sigman}^2-\Delta_h
+4\imath D_\mu+2\imath\gamma (\mu D_{\mu}-\overline{\sigman}/2-\imath (\dimnm)/4),
\end{split}\end{equation}
modulo terms that can be absorbed into the error terms in
operators in the class \eqref{eq:ah-final-conj-form}, i.e.\ this is indeed
of the form $P_\sigma^*+\imath Q_\sigma^*$ in the framework of
Subsection~\ref{subsec:complex-absorb-dS}, at least near $\mut=0$. If now $X'_0$
is extended to a manifold without boundary in such a way that in $\mut<0$,
i.e.\ $\mu>0$, one has a classically elliptic, semiclassically
non-trapping problem, then all the results of
Section~\ref{sec:microlocal} are applicable.
\end{rem}



\def\cprime{$'$} \def\cprime{$'$}

\end{document}